\documentclass[a4paper,11pt]{amsart}
\usepackage[utf8]{inputenc}
\usepackage[T1]{fontenc}
\usepackage[english]{babel}
\usepackage{amssymb,amsfonts,amsmath,amsthm}
\usepackage{mathtools}
\usepackage{graphicx}
\usepackage[margin=3.5cm]{geometry}
\usepackage[hidelinks]{hyperref}

\newtheorem{theorem}{Theorem}[section]
\numberwithin{equation}{section}
\newtheorem{proposition}[theorem]{Proposition}
\newtheorem{lemma}[theorem]{Lemma}
\newtheorem{corollary}[theorem]{Corollary}

\newtheorem{lettertheorem}{Theorem}

\theoremstyle{definition}

\theoremstyle{remark}
\newtheorem{remark}[theorem]{Remark}

\newcommand{\R}{\mathbb{R}}
\newcommand{\s}{\mathbb{S}}
\newcommand{\N}{\mathbb{N}}
\newcommand{\QQ}{\mathcal{Q}}
\newcommand{\Gpq}{O(p{+}1)\times O(q{+}1)}
\DeclareMathOperator{\sgn}{sgn}
\DeclareMathOperator{\dive}{div}
\DeclareMathOperator{\dist}{dist}

\title[Bi-rotational self-expanders]{The global structure of bi-rotational self-expanders}
\author{Gregório Silva Neto}
\address{Instituto de Matemática, Universidade Federal de Alagoas, Maceió, 57072-900, Brazil}
\email{gregorio@im.ufal.br}
\date{}
\subjclass[2020]{53E10; 53C42, 34C05}
\keywords{Mean curvature flow, self-expander, equivariant hypersurface,
asymptotically conical end, Clifford cone}
\begin{document}

\begin{abstract}
We describe the self-expanders of the mean curvature flow in $\R^{p+q+2}$ invariant under $\Gpq$. Apart from the Clifford minimal cone, every such hypersurface is complete, proper and noncompact, and belongs to one of three families: two one-parameter families of embedded hypersurfaces diffeomorphic to $\R^{p+1}\times\s^{q}$ or to $\R^{q+1}\times\s^{p}$, whose profile curves leave an axis, and a two-parameter family diffeomorphic to $\R\times\s^{p}\times\s^{q}$, parametrized by the point of the profile curve closest to the origin, which is the generic family: its set of initial conditions is open and of full measure in the phase space. Every end is asymptotically conical, with explicit decay rates, and every such hypersurface has Euclidean volume growth. A profile curve of the third family has at most one loop, and the point of the curve closest to the origin lies on that loop.
\end{abstract}

\maketitle
\tableofcontents

\section{Introduction}

An immersion $X\colon\Sigma^{n}\to\R^{n+1}$ is a self-similar solution of the mean curvature flow if
\begin{equation}\label{eq:soliton}
H=c\langle X,\nu\rangle,\qquad c\in\R\setminus\{0\}.
\end{equation}
For $c<0$ it is a self-shrinker, and models the singularities of the flow, while for $c>0$ it is a self-expander, and models the flow emerging from a conical singularity. This paper treats the case $c>0$. The case $c<0$ is treated in \cite{SN1}, whose notation we follow.

The theory of self-expanders is organized around the cone at infinity. Ecker and Huisken \cite{EH} proved asymptotic convergence to self-similarly expanding solutions for entire graphical flows under suitable hypotheses on the initial data. In particular, rotationally symmetric graphical cones give rise to rotationally symmetric graphical self-expanders. Angenent, Ilmanen and Chopp \cite{ACI} presented numerical evidence for multiple rotationally symmetric expanding evolutions associated with certain double cones, suggesting nonuniqueness of the flow emerging from a cone. Ding \cite{Ding} constructed smooth self-expanders asymptotic to certain mean-convex cones and to certain minimal cones that are not area-minimizing. Bernstein and Wang \cite{BW} developed the general theory of asymptotically conical self-expanders, and Fong and McGrath \cite{FM} proved that mean-convex expanders asymptotic to an $O(n)$-invariant cone are rotationally symmetric. On the shrinking side, Park \cite{Park-fol} constructed families of bi-rotational self-shrinking ends with boundary, and the full bi-rotational self-shrinker phase flow is studied in \cite{SN1}. Bi-rotational symmetry also appears in the study of ancient mean curvature flows asymptotic to minimal quadratic cones \cite{Park-Simons}.

The purpose of the present paper is to study the global behavior of the self-expanders invariant under the action of the group $O(p+1)\times O(q+1)$, about which very little is known. The symmetry reduces \eqref{eq:soliton} to a three-dimensional ordinary differential equation system, and it makes it possible to determine the global geometry of the whole bi-rotational class: the possible types of endpoint, the diffeomorphism type of each hypersurface, completeness and properness, the asymptotic cone of every end together with the rate at which it is approached, a bound on the self-intersections of the profile curve, and the behavior of the generic trajectory in the phase space. The usual approach prescribes a cone at infinity and constructs the expanders asymptotic to it. Here the phase portrait is described first, and the question of how many expanders are asymptotic to a given cone becomes a shooting problem for the limit angle.

Let $\Gpq$ act on $\R^{n+1}=\R^{p+1}\times\R^{q+1}$, $n+1=p+q+2$, $p,q\ge1$. An invariant hypersurface is parametrized by
\begin{equation}\label{eq:param}
X(s,\theta_1,\theta_2)=\left(x(s)\Psi_1(\theta_1), y(s)\Psi_2(\theta_2)\right),
\end{equation}
where $\Psi_1$ and $\Psi_2$ parametrize $\s^{p}$ and $\s^{q}$, and the profile curve $\gamma(s)=(x(s),y(s))$ lies in the phase plane $\QQ=\{x>0,y>0\}$. We call such a hypersurface bi-rotational, and we use this term from now
on, writing the group explicitly only where it matters. 

Parametrizing $\gamma$ by arclength we have $\gamma'(s)=(\cos\sigma(s),\sin\sigma(s))$, and equation \eqref{eq:soliton} is equivalent to the three-dimensional system
\begin{equation}\label{eq:flow}
\left\{
\begin{aligned}
x'&=\cos\sigma,\\
y'&=\sin\sigma,\\
\sigma'&=-A(x)\sin\sigma+B(y)\cos\sigma,
\end{aligned}\right.
\end{equation}
on $M:=\QQ\times\s^1$, where
\begin{equation}\label{eq:AB}    
A(x)=\frac px+cx\quad\mbox{and}\quad B(y)=\frac qy+cy.
\end{equation}
We also write 
\[x=\rho\cos\psi,\quad y=\rho\sin\psi, \quad\mbox{and}\quad w:=\psi-\sigma\] 
for the polar coordinates and the relative phase, respectively, and set $\psi_0:=\arctan\sqrt{q/p}$. The ray $\{\psi=\psi_0\}$, called the Clifford ray, is parametrized by
\[
(x,y,\sigma)=(s\cos\psi_0,s\sin\psi_0,\psi_0),
\]
and solves \eqref{eq:flow} for every $c$, generating the Clifford minimal cone
\[
\mathcal{C}=\{(u,v)\in \R^{p+1}\times\R^{q+1};|u|=\sqrt{p/q}|v|\},
\] 
over $\s^{p}(\sqrt{p/(p+q)})\times\s^{q}(\sqrt{q/(p+q)})$.

\subsection{Results}

The first theorem classifies the bi-rotational self-expanders by the behavior of their profile curves at the endpoints, determines the diffeomorphism type of each of them, and parametrizes the two-ended family by the point of the profile curve closest to the origin. The three types are illustrated in Figure~\ref{fig:types}.

\begin{figure}[b]
\centering
\includegraphics[width=\textwidth]{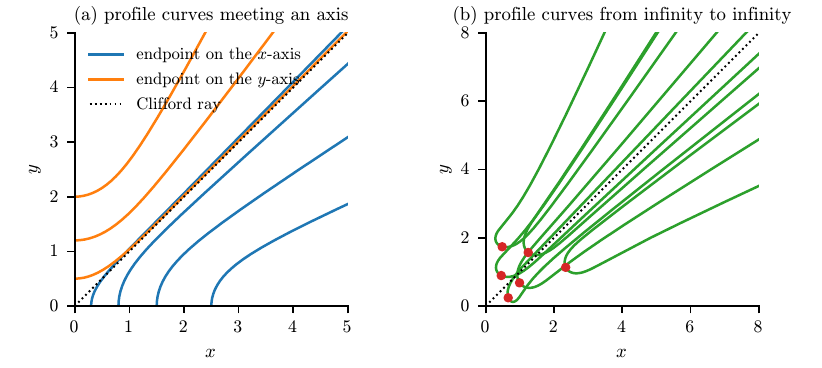}
\caption{The three types of profile curve of Theorem~\ref{thmA}, for $p=q=1$, $c=1$. In (a), curves of types (a) and (b), leaving the axes at $x_0\in\{0.3,0.8,1.5,2.5\}$ and $y_0\in\{0.5,1.2,2\}$. In (b), curves of type (c), with their closest points to the origin marked, at polar coordinates $(\rho,\psi)\in\{(1.8,1.3),(2.0,0.9),(1.0,1.1),(1.2,0.6),(0.7,0.35),(2.6,0.45)\}$. None of these six has a self-intersection. The dotted line is the Clifford ray, drawn for reference only: by Theorem~\ref{thmB} each end has an asymptotic ray of its own.}
\label{fig:types}
\end{figure}

\begin{lettertheorem}[Classification]\label{thmA}
Let $\Sigma$ be a bi-rotational self-expander associated with a maximal solution of \eqref{eq:flow} and different from the Clifford cone. Then $\Sigma$ is complete and noncompact, and its profile curve is of one of the following types:
\begin{enumerate}
\item[(a)] Endpoint on the $x$-axis: the curve leaves the $x$-axis orthogonally at a point $(x_0,0)$, the tangent angle satisfies $\sigma\in(0,\tfrac\pi2)$ along the whole curve, and $x$, $y$ and $\rho$ are strictly increasing. Consequently, $\Sigma$ is embedded and diffeomorphic to $\R^{q+1}\times\s^{p}$, and has one end. For every $x_0>0$ there is exactly one such curve.
\item[(b)] Endpoint on the $y$-axis: the mirror image of (a). Here, $\Sigma$ is embedded and diffeomorphic to $\R^{p+1}\times\s^{q}$.
\item[(c)] From infinity to infinity: both endpoints satisfy $\rho\to\infty$, and $\rho$ has exactly one critical point along the curve, a nondegenerate minimum. In this case, $\Sigma$ is diffeomorphic to $\R\times\s^{p}\times\s^{q}$, and has two ends.
\end{enumerate}
Moreover, every point of $\QQ$ is the closest point to the origin of exactly one curve of type (c), up to orientation.
\end{lettertheorem}

The second theorem describes the ends. For clarity, in the bi-rotational setting, an end is called asymptotically conical when its profile curve approaches a ray and its tangent becomes radial. Equivalently, the dilations of $\Sigma$ converge in $C^{1}$, on compact annuli away from the vertex, to the invariant cone generated by that ray. The theorem below gives the quantitative form of this convergence, with the rates $|\sin w|=O(\rho^{-2})$ and $|\psi-\psi_\infty|=O(\rho^{-2})$, which are what the later sections use.

\begin{lettertheorem}[Ends]\label{thmB}
Let $\gamma$ be a profile curve  of a bi-rotational self-expander  $\Sigma$ distinct from the Clifford cone. Along each end of $\gamma$, the following hold:
\begin{enumerate}
\item[(i)] $x\to\infty$ and $y\to\infty$;
\item[(ii)] the polar angle converges to some $\psi_\infty\in(0,\tfrac\pi2)$, the tangent becomes radial,
\[
|\sin w|=O(\rho^{-2}),\quad |\psi-\psi_\infty|=O(\rho^{-2}),\quad\mbox{and}\quad |\mathcal A|^{2}=O(\rho^{-2}).
\]
In particular, the profile curve lies at distance $O(\rho^{-1})$ from the ray $\{\psi=\psi_\infty\}$;
\item[(iii)] $x$ and $y$ are eventually strictly increasing, and the end is a graph over each coordinate axis, with slopes converging to $\tan\psi_\infty$ and $\cot\psi_\infty$;
\item[(iv)] if $\psi_\infty\ne\psi_0$, then $\psi$ is eventually strictly monotone, and the end approaches the ray $\{\psi=\psi_\infty\}$ from the side away from the Clifford ray;
\item[(v)] If $\psi\ne\psi_0$ along the whole curve, then, when moving away from the axis, $\psi$ is strictly increasing along a curve of type (a) and strictly decreasing along a curve of type (b). Along a curve of type (c), $\psi$ has exactly one critical point on the whole curve, which is a nondegenerate global minimum if $\psi<\psi_0$, and a nondegenerate global maximum if $\psi>\psi_0$. Consequently, the mean curvature of $\Sigma$ with respect to $\nu$ is positive along a curve of type (a), negative along a curve of type (b), and changes sign exactly once along a curve of type (c).
\end{enumerate}
Consequently, $\Sigma$ is properly immersed. If $\psi_1,\ldots,\psi_k$ are the asymptotic angles of its $k$ ends, where $k=1$ for types (a) and (b), and $k=2$ for type (c), then
\[
\lim_{R\to\infty}\frac{\mathrm{vol}(\Sigma\cap B_R)}{R^n}
=\frac{\omega_p\omega_q}{n}\sum_{j=1}^{k}\cos^p\psi_j\,\sin^q\psi_j>0,
\]
where $n=p+q+1$, $\omega_p=\mathrm{vol}(\s^p)$ and $\omega_q=\mathrm{vol}(\s^q)$.
\end{lettertheorem}

\begin{remark}
    It is a consequence of Theorem \ref{thmB} that no end of a bi-rotational self-expander has one coordinate bounded while the other diverges, so that no end is asymptotic to a cylinder $\s^{q}(r)\times\R^{p+1}$ or $\s^{p}(r)\times\R^{q+1}$, nor to a coordinate subspace.
\end{remark}

The third theorem bounds the self-intersections of a profile curve. By a loop we mean a simple closed subarc $\gamma|_{[a,b]}$, with $\gamma(a)=\gamma(b)$ and distinct tangents at $a$ and $b$. By Theorem \ref{thmA}, the profile curves of hypersurfaces of types (a) and (b) are simple (since the hypersurfaces are embedded). The behavior of the profile curves of type (c) is the content of the following

\begin{lettertheorem}[Loops]\label{thmC}
Let $\Sigma$ be a bi-rotational self-expander with a profile curve $\gamma$ of type (c) of Theorem \ref{thmA}. Then $\gamma$ has at most one loop $\gamma|_{[a,b]}$, and if it exists,
\begin{enumerate}
\item[(i)] the point of $\gamma$ closest to the origin lies on $\gamma((a,b))$, and $\gamma((a,b))\subset\{x<x_P, y<y_P\}$, where $P=(x_P,y_P)=\gamma(a)=\gamma(b)$;
\item[(ii)] the turning of the tangent angle along the loop satisfies
\[
\pi<|\sigma(b)-\sigma(a)|<\frac{3\pi}2;
\]
\item[(iii)] the region $\Omega$ bounded by the loop is not contained in $[\sqrt{p/c},\infty)\times[\sqrt{q/c},\infty)$, so that the loop meets the strip $\{x<\sqrt{p/c}\}$ or the strip $\{y<\sqrt{q/c}\}$.
\end{enumerate}
Moreover, along $\gamma$, the functions $\sin\sigma$ and $\cos\sigma$ change sign exactly once each, from $-$ to $+$, and the turning of the tangent angle along the whole curve is
\[
\lim_{s\to s_+}\sigma-\lim_{s\to s_-}\sigma=\psi_\omega-\psi_\alpha\pm\pi ,
\]
where $\psi_\alpha$ and $\psi_\omega$ are the limit angles of the two ends.
\end{lettertheorem}

The limit angles $\psi_\alpha$ and $\psi_\omega$ are the extreme values of a monotone angular quantity along $\gamma$ (Corollary~\ref{cor:extremos}).

Our last theorem says that curves of type (c) are the typical ones, in the sense that they form a set of full measure in the sense we introduce below. Let 
\[
\mathfrak m:=\Phi dx dy d\sigma,\quad\mbox{where}\quad\Phi=x^{p}y^{q}e^{c(x^{2}+y^{2})/2},
\]
a measure on the phase space $M$.

\begin{lettertheorem}[Genericity]\label{thmD}
Let $c>0$. The flow of \eqref{eq:flow} preserves $\mathfrak m$. The set of initial conditions whose trajectory is of type (c) is open and has full $\mathfrak m$-measure, and its complement, formed by the trajectories of types (a), (b) and the Clifford ray, is $\mathfrak m$-null.
\end{lettertheorem}

\begin{remark}
Since $\Phi$ is smooth and positive, $\mathfrak m$ and the Lebesgue measure of $M$ are mutually absolutely continuous, so that set also has full Lebesgue measure. Its complement is Lebesgue-null and therefore has empty interior, so the set is dense in $M$.
\end{remark}

\subsection{The method}
For $c>0$ the functions $A$ and $B$ defined by \eqref{eq:AB} are positive on $\QQ$. This single fact is what separates the expanding case from the shrinking one: the two functions
\[
I=x^{p}e^{cx^{2}/2}\sin\sigma,\qquad J=y^{q}e^{cy^{2}/2}\cos\sigma
\]
of \cite{SN1}, which are only piecewise monotone for $c<0$, are strictly increasing along every trajectory. There are no cylinders, no winding, no closed orbits and no recurrence, and the analysis is global. From these monotone quantities we derive the behavior at the finite endpoints, the quantitative conical asymptotics of the ends, the diffeomorphism types, the bound on the self-intersections and the genericity of the two-ended family.

\subsection{Organization}
Section~\ref{sec:system} sets up the profile system and the monotone quantities on which everything rests. Section~\ref{sec:endpoints} describes the endpoints, proves completeness and proves Theorem~\ref{thmA}. Section~\ref{sec:ends} proves Theorem~\ref{thmB}, and Section~\ref{sec:loops} proves Theorems~\ref{thmC} and~\ref{thmD} and discusses the questions left open by the classification. The local facts about finite endpoints used in Section~\ref{sec:endpoints} are proved in \cite{SN1} for every $c\ne0$.

\subsection*{Acknowledgements}
The author was partially supported by CNPq, Brazil, grant number 306189/2023-2.

\subsection*{AI statement}
Generative AI tools (Claude, Anthropic, and ChatGPT, OpenAI) were used to assist with language editing, organization of the exposition, numerical experiments, suggestions for some arguments, and the discussion and checking of selected calculations. All AI-assisted material was reviewed and edited by the author, who takes full responsibility for the mathematical statements, proofs, and references.

\subsection*{Data availability statement}
No data are associated with this manuscript. Figures~\ref{fig:types}, \ref{fig:lado} and~\ref{fig:loop} and Remark~\ref{rem:lado-numerico} are based on numerical solutions of \eqref{eq:flow} with the parameters and initial data given in the captions and in the remark. These computations are not certified, and no proof relies on them.

\subsection*{Conflict of interest statement}
The author declares that he has no conflict of interest.

\section{The profile system and its monotone quantities}\label{sec:system}

\subsection{The system, its symmetries and its polar form}\label{sub:system}

We begin by writing the soliton equation as a system for the profile curve. The computation is the standard one in equivariant geometry, and we record it together with the second fundamental form, which is used in Section~\ref{sec:ends}.

\begin{proposition}\label{prop:sistema}
Let $\Sigma$ be a hypersurface invariant under $\Gpq$, parametrized by \eqref{eq:param}, whose profile curve $\gamma(s)=(x(s),y(s))$ is parametrized by arclength, and write $\gamma'(s)=(\cos\sigma(s),\sin\sigma(s))$. Then
\[
\nu=\left(\Psi_1\sin\sigma,-\Psi_2\cos\sigma\right)
\]
is a unit normal, the principal curvatures are $-\sigma'$, $-\sin\sigma/x$ and $\cos\sigma/y$ with multiplicities $1$, $p$ and $q$,
\[
H=-\sigma'-\frac{p\sin\sigma}{x}+\frac{q\cos\sigma}{y},\quad
\langle X,\nu\rangle=x\sin\sigma-y\cos\sigma,
\]
and
\begin{equation}\label{eq:A2}
|\mathcal A|^{2}=(\sigma')^{2}+p\frac{\sin^{2}\sigma}{x^{2}}+q\frac{\cos^{2}\sigma}{y^{2}} .
\end{equation}
Consequently, $\Sigma$ satisfies \eqref{eq:soliton} if and only if $\gamma$ solves \eqref{eq:flow}.
\end{proposition}
\begin{proof}
Differentiating \eqref{eq:param}, the tangent space is spanned by
\[
\partial_sX=(\cos\sigma\Psi_1,\sin\sigma\Psi_2),
\]
by the vectors $(x\partial_i\Psi_1,0)$ and by the vectors $(0,y\partial_j\Psi_2)$, which are mutually orthogonal, and $\nu$ is orthogonal to all of them and has unit length. From \[\partial_s\nu=\sigma'\partial_sX,\quad \partial_i\nu=(\sin\sigma\partial_i\Psi_1,0)=\frac{\sin\sigma}{x}(x\partial_i\Psi_1,0),\] and \[\partial_j\nu=(0,-\cos\sigma\partial_j\Psi_2)=-\frac{\cos\sigma}{y}(0,y\partial_j\Psi_2)\] we read off the principal curvatures of the shape operator $-d\nu$, and hence $H$ and \eqref{eq:A2}. Finally
\[
H-c\langle X,\nu\rangle=-\left[\sigma'+\left(\frac px+cx\right)\sin\sigma-\left(\frac qy+cy\right)\cos\sigma\right],
\]
which vanishes exactly when $\gamma$ solves \eqref{eq:flow}.
\end{proof}

The system has two symmetries, which will spare us half of the work: reversing the orientation of the profile curve, and exchanging the two factors of the group.

\begin{lemma}[Symmetries]\label{lem:simetrias}
Let $(x,y,\sigma)$ solve \eqref{eq:flow} with parameters $(p,q)$. Then
\begin{enumerate}
\item[(i)] $s\mapsto\left(x(-s),y(-s),\sigma(-s)+\pi\right)$ solves \eqref{eq:flow} with the same parameters;
\item[(ii)] $s\mapsto\left(y(s),x(s),\tfrac\pi2-\sigma(s)\right)$ solves \eqref{eq:flow} with parameters $(q,p)$.
\end{enumerate}
\end{lemma}
\begin{proof}
Both are direct substitutions. In (i), the derivatives of the first two components change sign, and so do $\cos(\sigma+\pi)$ and $\sin(\sigma+\pi)$, while the third equation is odd under $s\mapsto-s$ together with $\sigma\mapsto\sigma+\pi$. In (ii), $\tilde x'=y'=\sin\sigma=\cos\tilde\sigma$ and $\tilde y'=x'=\cos\sigma=\sin\tilde\sigma$, and $\tilde\sigma'=-\sigma'$ becomes $-\tilde A\sin\tilde\sigma+\tilde B\cos\tilde\sigma$ with $\tilde A(t)=q/t+ct$ and $\tilde B(t)=p/t+ct$, which are the coefficients for the parameters $(q,p)$.
\end{proof}
The next lemma gives an equivalent form of system \eqref{eq:flow} in polar coordinates.
\begin{lemma}[Polar form]\label{lem:polar}
In the polar coordinates $x=\rho\cos\psi$, $y=\rho\sin\psi$, with $w=\psi-\sigma$, a solution of \eqref{eq:flow} satisfies
\begin{equation}\label{eq:polar}
\left\{
\begin{aligned}\rho'&=\cos w,\\ \psi'&=-\frac{\sin w}{\rho},\\
\sigma'&=c\rho\sin w+\frac{E(\psi,\sigma)}{\rho},
\end{aligned}\right.
\end{equation}
where $E(\psi,\sigma):=\dfrac{q\cos\sigma}{\sin\psi}-\dfrac{p\sin\sigma}{\cos\psi}$.
\end{lemma}
\begin{proof}
One has $\rho'=(xx'+yy')/\rho=\cos\psi\cos\sigma+\sin\psi\sin\sigma=\cos w$ and $\psi'=(xy'-yx')/\rho^{2}=-\sin w/\rho$. For the third equation, group the terms of \eqref{eq:flow} containing $c$:
\[
\sigma'=c(-x\sin\sigma+y\cos\sigma)+\frac{q\cos\sigma}{y}-\frac{p\sin\sigma}{x}
=c\rho\sin w+\frac{E(\psi,\sigma)}{\rho},
\]
since $-x\sin\sigma+y\cos\sigma=\rho\sin(\psi-\sigma)$.
\end{proof}

\subsection{Monotone quantities}\label{sec:mono}

The analysis rests on a few quantities that are monotone along every trajectory. The first two are built from the weights
\[
f(x):=x^{p}e^{cx^{2}/2},\qquad g(y):=y^{q}e^{cy^{2}/2},
\]
which satisfy $f'=Af$ and $g'=Bg$, and are
\begin{equation}\label{eq:IJ}
I:=f(x)\sin\sigma,\qquad J:=g(y)\cos\sigma .
\end{equation}
They come from splitting the right-hand side of the equation for $\sigma$ in \eqref{eq:flow} into a term depending only on $(\sigma,x)$ and a term depending only on $(\sigma,y)$. Indeed, $I$ is a first integral of the planar system $\sigma'=-A(x)\sin\sigma$, $x'=\cos\sigma$, and $J$ is a first integral of $\sigma'=B(y)\cos\sigma$, $y'=\sin\sigma$, as in \cite[Section~2.2]{SN1}. They are not conserved by \eqref{eq:flow}, but their derivatives along it have the signs of $B$ and $A$, as the next lemma shows. Since
\[
\frac{I^{2}}{f^{2}}+\frac{J^{2}}{g^{2}}=1,
\]
the pair $(I,J)$ determines the tangent angle. The zeros of $I$ are the critical points of $y$, and the zeros of $J$ are those of $x$.

\begin{lemma}[The two functions]\label{lem:IJ}
Along every solution of \eqref{eq:flow},
\begin{equation}\label{eq:dJ}
I'=B(y) f(x)\cos^{2}\sigma,\qquad J'=A(x) g(y)\sin^{2}\sigma .
\end{equation}
In particular, $I'$ has the sign of $B(y)$ and $J'$ has the sign of $A(x)$, and both vanish only where the profile curve is vertical, respectively horizontal.
\end{lemma}
\begin{proof}
Using $f'=Af$, $x'=\cos\sigma$ and the third equation of \eqref{eq:flow},
\[
\begin{aligned}
I'&=f'x'\sin\sigma+f\sigma'\cos\sigma\\
&=Af\cos\sigma\sin\sigma+f\cos\sigma\left(-A\sin\sigma+B\cos\sigma\right)\\
&=Bf\cos^{2}\sigma.
\end{aligned}
\]
The computation for $J$ is the mirror image.
\end{proof}

The two signs in \eqref{eq:dJ} are what separate the expanding case from the shrinking one. For $c<0$ the coefficients $A$ and $B$ vanish at $x_*=\sqrt{-p/c}$ and $y_*=\sqrt{-q/c}$ and change sign there, so $I$ and $J$ are only piecewise monotone, and the two lines $\{x=x_*\}$ and $\{y=y_*\}$, which are themselves solutions, split the phase plane into four regions of distinct behavior, and this is the source of the winding and of the recurrence found in \cite{SN1}. For $c>0$ no such splitting occurs, as the next lemma records.

\begin{lemma}[Positivity]\label{lem:positivo}
Let $c>0$. On $\QQ$,
\[
A(x)\ge2\sqrt{pc}>0,\qquad B(y)\ge2\sqrt{qc}>0,
\]
with equality at $x=\sqrt{p/c}$ and $y=\sqrt{q/c}$. The functions $f$ and $g$ are strictly increasing and tend to $0$ at $0$ and to $\infty$ at $\infty$.
\end{lemma}
\begin{proof}
The inequality is the arithmetic--geometric mean inequality applied to $p/x$ and $cx$. Since $f'=Af>0$, $f$ is increasing, and $\log f=p\log x+\frac c2x^{2}$ tends to $-\infty$ at $0$ and to $+\infty$ at $\infty$. The same holds for $g$.
\end{proof}

Positivity of $A$ and $B$ turns $I$ and $J$ into globally monotone quantities, which is the source of the rigidity of the expanding case.

\begin{corollary}[Global monotonicity]\label{cor:mono}
Let $c>0$. Along every solution of \eqref{eq:flow}, $I$ and $J$ are strictly increasing. Consequently each of $\sin\sigma$ and $\cos\sigma$ vanishes at most once, and changes sign, if at all, from $-$ to $+$. In particular $x$ and $y$ are monotone outside at most one instant each.
\end{corollary}
\begin{proof}
By \eqref{eq:dJ} and Lemma~\ref{lem:positivo}, $I'\ge0$, with equality only where $\cos\sigma=0$. At such a point $\sin\sigma=\pm1$ and $\sigma'=-A(x)\sin\sigma$, so that
\[
\frac{d}{ds}\cos\sigma=-\sin\sigma\,\sigma'=A(x)\sin^{2}\sigma=A(x)>0 ,
\]
and the zeros of $\cos\sigma$ are isolated, each of them being a passage from $-$ to $+$. Likewise, at a zero of $\sin\sigma$ one has $\frac{d}{ds}\sin\sigma=B(y)>0$. A continuous function with nonnegative derivative vanishing only at isolated points is strictly increasing. The same argument applies to $J$, with $B$ and $\sin\sigma$. Since $f,g>0$, $\sgn I=\sgn\sin\sigma$ and $\sgn J=\sgn\cos\sigma$.
\end{proof}

The third quantity involves the weight $\Phi:=f(x)g(y)$, for which $\partial_x\log\Phi=A$ and $\partial_y\log\Phi=B$.

\begin{lemma}[A third monotone quantity]\label{lem:terceira}
Let $c\in\R$, $n:=p+q+1$, and set $F:=\Phi\langle\gamma,\gamma'\rangle=\Phi\rho\rho'$ along a solution of \eqref{eq:flow}. Then
\begin{equation}\label{eq:F}
F'=\Phi(n+c\rho^{2}).
\end{equation}
\end{lemma}
\begin{proof}
From $\Phi'=\Phi(A\cos\sigma+B\sin\sigma)$, $\frac{d}{ds}(x\cos\sigma+y\sin\sigma)=1+\sigma'(-x\sin\sigma+y\cos\sigma)$ and \eqref{eq:flow},
\[
\begin{aligned}
\frac{F'}{\Phi}&=(A\cos\sigma+B\sin\sigma)(x\cos\sigma+y\sin\sigma)\\
&\quad+1+(-A\sin\sigma+B\cos\sigma)(-x\sin\sigma+y\cos\sigma)\\
&=1+Ax+By=1+p+q+c\rho^{2}.
\end{aligned}
\]
\end{proof}

Since $\Phi>0$ and $n+c\rho^{2}>0$ for $c>0$, the quantity $F$ is strictly increasing along every solution. As $F$ has the sign of $\rho'$, this controls the distance to the origin.

\begin{corollary}[Unimodality of $\rho$]\label{cor:unimodal}
Let $c>0$. Along every solution of \eqref{eq:flow}, $\rho'$ changes sign at most once, from $-$ to $+$. Hence $\rho$ has at most one critical point, which is a strict global minimum, and at that point $(\rho^{2}/2)''=n+c\rho^{2}>0$.
\end{corollary}
\begin{proof}
By \eqref{eq:F}, $F$ is strictly increasing, and $\sgn F=\sgn\rho'$. At a critical point of $\rho$ we have $F=0$ and $F'=\Phi(\rho\rho')'$, so that $(\rho^{2}/2)''=F'/\Phi=n+c\rho^{2}$.
\end{proof}

\begin{remark}\label{rem:laplaciano}
Corollary~\ref{cor:unimodal} is the equivariant form of the identity $\Delta_\Sigma\frac{|X|^{2}}2=n+c|X^{\perp}|^{2}$, valid on every self-expander: $|X|^{2}$ has no local maximum on $\Sigma$.
\end{remark}

The angular part of $F$ is monotone as well, and it is what confines the polar angle. Set
\[
h(\psi):=\cos^{p}\psi\sin^{q}\psi,\qquad \psi\in(0,\tfrac\pi2),
\]
so that $\Phi=\rho^{p+q}h(\psi)e^{c\rho^{2}/2}$ and, by \eqref{eq:polar}, $F=\rho^{n}e^{c\rho^{2}/2} h(\psi)\cos w$. Since $(\log h)'=q\cot\psi-p\tan\psi$, the function $h$ is positive on $(0,\frac\pi2)$, tends to $0$ at both ends, and attains its maximum exactly at the Clifford angle $\psi_0$.

\begin{lemma}[The angular quantity]\label{lem:L}
Let $c>0$ and set $L:=h(\psi)\cos w$ along a solution of \eqref{eq:flow}. Then
\begin{equation}\label{eq:L}
\rho\frac{dL}{ds}=(n+c\rho^{2}) h(\psi)\sin^{2}w\ge0 ,
\end{equation}
so $L$ is nondecreasing, and strictly increasing except at the instants where the tangent is radial.
\end{lemma}
\begin{proof}
Write $F=\rho^{n}e^{c\rho^{2}/2}L$ and differentiate, using $\rho'=\cos w$:
\[
\rho F'=(n+c\rho^{2})\cos w \rho^{n}e^{c\rho^{2}/2}L+\rho^{n}e^{c\rho^{2}/2} \rho L' .
\]
On the other hand $\rho F'=(n+c\rho^{2})\rho\Phi=(n+c\rho^{2})\rho^{n}e^{c\rho^{2}/2}h$ by \eqref{eq:F}. Hence $\rho L'=(n+c\rho^{2})(h-L\cos w)=(n+c\rho^{2})h\sin^{2}w$.
\end{proof}

\begin{corollary}[Confinement of the polar angle]\label{cor:confina}
Let $c>0$ and let $\gamma$ be a solution of \eqref{eq:flow}. If $L(s_0)>0$ for some $s_0$, then
\[
h(\psi(s))\ge L(s)\ge L(s_0)>0\qquad\text{for every }s\ge s_0 ,
\]
and $\psi$ stays in a compact subinterval of $(0,\frac\pi2)$ on $[s_0,\infty)$. Symmetrically, if $L(s_0)<0$, then $h(\psi(s))\ge|L(s_0)|$ for every $s\le s_0$.
\end{corollary}
\begin{proof}
By Lemma~\ref{lem:L}, $L(s)\ge L(s_0)$ for $s\ge s_0$, and $h(\psi)\ge h(\psi)\cos w=L$. Since $h$ is continuous and positive on $(0,\frac\pi2)$ and vanishes at the endpoints, the set $\{h\ge L(s_0)\}$ is a compact subinterval of $(0,\frac\pi2)$. If $L(s_0)<0$ and $s\le s_0$, then $L(s)\le L(s_0)$, hence $h(\psi)|\cos w|\ge|L(s_0)|$ and again $h(\psi)\ge|L(s_0)|$.
\end{proof}

\section{Endpoints, completeness and classification}\label{sec:endpoints}

\subsection{The three kinds of endpoint}\label{sub:endpoints}

The system \eqref{eq:flow} is smooth on $M$, so every solution extends uniquely to a maximal one, defined on an open interval $(s_-,s_+)$ with $-\infty\le s_-<s_+\le\infty$. We call $\alpha$ and $\omega$ the two endpoints of the trajectory, at $s_-$ and at $s_+$, respectively. They refer to the ends of the parameter interval, not to points of $M$ or of the plane. Thus a limit at $\omega$ is a limit as $s\to s_+$, and a property holds near $\omega$ if it holds on some interval $(s_1,s_+)$, and similarly for $\alpha$. When the profile curve converges as $s\to s_+$, we also denote its limit point by $\omega$, and similarly for $\alpha$. The reversal of Lemma~\ref{lem:simetrias}(i) exchanges $\alpha$ and $\omega$. The notation is borrowed from the $\alpha$- and $\omega$-limit sets of dynamical systems, which describe the same two directions of time.

We treat the endpoint $\omega$. The statements for $\alpha$ follow by the reversal.

\begin{proposition}[Finite and infinite endpoints]\label{prop:endpoints}
Let $c>0$. Exactly one of the following holds:
\begin{enumerate}
\item[(i)] $s_+<\infty$, and $\gamma(s)$ converges, as $s\to s_+$, to a point of $\partial\QQ$.
\item[(ii)] $s_+=\infty$ and $\rho\to\infty$. More precisely, $x$ and $y$ are eventually monotone and at least one of them tends to $\infty$.
\end{enumerate}
In particular, there is no end of infinite length along which $\gamma$ stays bounded.
\end{proposition}
\begin{proof}
By Corollary~\ref{cor:mono}, $x$ and $y$ are monotone on some interval $(s_1,s_+)$. If both are bounded there, they converge, and since $|x'|+|y'|\ge1$, the length of $\gamma|_{(s_1,s_+)}$ is at most the sum of their total variations, hence $s_+<\infty$. In that case, since $|\gamma'|=1$,
\[
|\gamma(s)-\gamma(r)|=\left|\int_r^{s}\gamma'(t) dt\right|\le|s-r|\qquad\text{for }r,s\in(s_1,s_+),
\]
so $\gamma$ is $1$-Lipschitz on a bounded interval, hence uniformly continuous, and extends continuously to $s_+$, converging to some $P\in\overline{\QQ}$, and $P\notin\QQ$, because a maximal solution with $s_+<\infty$ leaves every compact subset of $M$ \cite[Ch.~II, \S3]{Hartman}. If one of them is unbounded, it tends to $\infty$ monotonically, and $s_+=\infty$ because the speed is $1$.
\end{proof}

The next result describes the behavior at a finite endpoint. It is proved in \cite[Propositions~2.16, 6.1 and~6.4]{SN1} for every $c\ne0$.

\begin{proposition}[Behavior at a finite endpoint]\label{prop:sn1}
Let $c\ne0$ and let $\gamma$ be a maximal solution of \eqref{eq:flow} whose profile curve converges, as $s\to s_+$ or $s\to s_-$, to a point $P\in\partial\QQ$.
\begin{enumerate}
\item[(i)] If $P=(x_0,0)$ with $x_0>0$, then the curve meets the axis orthogonally: $\sin\sigma\to-1$ at $\omega$ and $\sin\sigma\to+1$ at $\alpha$, with $I\to\mp f(x_0)$. Symmetrically if $P=(0,y_0)$ with $y_0>0$.
\item[(ii)] In that case $\Sigma$ closes off smoothly over $\s^{p}(x_0)\times\{0\}$, which is a regular interior subset of $\Sigma$.
\item[(iii)] If $P$ is the origin, then $\gamma$ is the Clifford ray.
\end{enumerate}
\end{proposition}

The two propositions together show that an endpoint is of exactly one of three kinds: an orthogonal endpoint on one of the two axes, the origin, which occurs only for the Clifford ray, or an endpoint at infinite arclength, along which $\rho\to\infty$. So, completeness follows:

\begin{corollary}[Completeness]\label{thm:completude}
Let $c>0$. Every hypersurface associated with a maximal solution of \eqref{eq:flow}, other than the Clifford cone, is complete. The Clifford cone is incomplete, its profile curve reaching the origin in finite arclength.
\end{corollary}
\begin{proof}
By Proposition~\ref{prop:endpoints} and Proposition~\ref{prop:sn1}, each endpoint is at infinite arclength, an orthogonal endpoint on an axis, over which $\Sigma$ closes off smoothly, or the origin, which is reached only by the Clifford ray.
\end{proof}

\subsection{Classification}\label{sec:classif}

We first describe the curves with an endpoint on an axis, which are the simplest ones: along them the tangent angle never leaves the first quadrant.

\begin{proposition}[Curves leaving an axis]\label{prop:cusp}
Let $c>0$ and suppose that the endpoint $\alpha$ of $\gamma$ lies on the $x$-axis, at $(x_0,0)$. Then 
\begin{itemize}
    \item[(i)] $\sigma\in(0,\tfrac\pi2)$ along the whole curve;
    \item[(ii)] $x$, $y$ and $\rho$ are strictly increasing;
    \item[(iii)] the endpoint $\omega$ satisfies $\rho\to\infty$;
    \item[(iv)] and $\gamma$ is simple.
\end{itemize}
    The same holds, mutatis mutandis, for an endpoint on the $y$-axis, and when the endpoint on the axis is $\omega$ the corresponding statements hold after reversing the orientation, that is, with all the monotonicity assertions read as one moves away from the axis.
\end{proposition}
\begin{proof}
At $\alpha$, $I\to f(x_0)>0$ by Proposition~\ref{prop:sn1}(i) and $J\to0$, since $|J|\le g(y)\to0$. By Corollary~\ref{cor:mono}, $I>f(x_0)>0$ and $J>0$ along the curve, so $\sin\sigma>0$ and $\cos\sigma>0$. Hence $x$ and $y$ are strictly increasing, and so is $\rho$, since $\rho\rho'=x\cos\sigma+y\sin\sigma>0$. The endpoint $\omega$ is neither on an axis nor at the origin, since $x$ and $y$ increase, and Proposition~\ref{prop:endpoints} gives $\rho\to\infty$ there. Since $x$ is strictly increasing, $\gamma$ is injective. When the endpoint on the axis is $\omega$ we apply reversibility, and for the $y$-axis, the mirror symmetry of Lemma~\ref{lem:simetrias}(ii).
\end{proof}

\begin{remark}\label{rem:nocompact}
There is no compact self-expander, since by the identity of Remark~\ref{rem:laplaciano} the function $|X|^{2}$ would attain a maximum. In the bi-rotational case this also follows from Corollary~\ref{cor:mono}, since $I$ is strictly increasing and hence excludes closed orbits, and from Proposition~\ref{prop:cusp}, which excludes profile curves with both endpoints on the axes.
\end{remark}

The curves of Proposition~\ref{prop:cusp} are determined by their endpoint on the axis. This does not follow from the standard existence and uniqueness theorem, since the coefficient $B(y)=q/y+cy$ of \eqref{eq:flow} is singular at $y=0$. Near the axis we write the curve as a graph over the $y$-axis, and in divergence form the equation becomes a fixed-point problem in which the singular term is harmless.

\begin{proposition}[Existence and uniqueness at an axis]\label{ap:existencia}
Let $c\ne0$ and $x_0>0$. Up to a translation of the parameter, there is exactly one maximal solution of \eqref{eq:flow} whose endpoint $\alpha$ is $(x_0,0)$. Near that endpoint its profile curve is a graph $x=X(y)$, $0<y\le\delta$, with $X(y)\to x_0$ and $X'(y)=O(y)$. The mirror statement holds on the $y$-axis, with $p$ in place of $q$.
\end{proposition}
\begin{proof}
On an arc where $\sin\sigma>0$, the curve is a graph $x=X(y)$, and $V:=X'=\cot\sigma$ satisfies, by \eqref{eq:flow},
\[
\frac{dV}{dy}=-\frac{\sigma'}{\sin^{3}\sigma}=\left(1+V^{2}\right)\left(A(X)-B(y)V\right).
\]
Equivalently,
\begin{equation}\label{ap:eq:div}
\left(y^{q}V\right)'=y^{q}G(y,X,V),\qquad G(y,X,V):=(1+V^{2})\left(A(X)-cyV\right)-\frac qyV^{3}.
\end{equation}

\emph{Existence.} Let $a:=\max\{|A(x)|:|x-x_0|\le x_0/2\}$ and $K:=1+2a/(q+1)$. For $0<\delta\le1$, let $E_\delta$ be the set of pairs $(X,V)$ of continuous functions on $[0,\delta]$ with $|X-x_0|\le x_0/2$ and $|V(t)|\le Kt$, endowed with the norm $\|(X,V)\|:=\max\left(\sup|X|,\sup|V|\right)$. It is a closed subset of a Banach space, hence complete. Define $T(X,V):=(\tilde X,\tilde V)$ by
\[
\tilde V(y):=y^{-q}\int_0^{y}t^{q}G\left(t,X(t),V(t)\right)\,dt,\qquad \tilde X(y):=x_0+\int_0^{y}V(t)\,dt ,
\]
with $\tilde V(0):=0$. For $(X,V)\in E_\delta$ and $0<t\le\delta$ one has $|G(t,X,V)|\le a+C_1t^{2}$, with $C_1$ depending only on $p,q,c,x_0$. Hence $|\tilde V(y)|\le y(a+C_1\delta^{2})/(q+1)\le Ky$ and $|\tilde X-x_0|\le K\delta^{2}/2\le x_0/2$ for $\delta$ small, so $T(E_\delta)\subset E_\delta$. On $E_\delta$ the function $G$ is Lipschitz in $(X,V)$, with a constant $\Lambda$ independent of $\delta\le1$, because $A$ is smooth on $[x_0/2,3x_0/2]$ and the term $qV^{3}/t$ has Lipschitz constant at most $3qK^{2}t$ in $V$. Therefore
\[
\|T(X_1,V_1)-T(X_2,V_2)\|\le\left(\frac{\Lambda}{q+1}+1\right)\delta\,\|(X_1,V_1)-(X_2,V_2)\|,
\]
and $T$ is a contraction for $\delta$ small. Its fixed point $(X,V)$ satisfies $X'=V$ and \eqref{ap:eq:div} on $(0,\delta]$, with $X\in C^{2}((0,\delta])$. Defining $\sigma\in(0,\pi)$ by $\cot\sigma=V$ and parametrizing the graph by arclength, we obtain a solution of \eqref{eq:flow}, which we extend to a maximal one. Since $X\to x_0$ as $y\to0$, its endpoint $\alpha$ is $(x_0,0)$.

\emph{Uniqueness.} Let $\gamma$ be a maximal solution whose endpoint $\alpha$ is $(x_0,0)$. By Proposition~\ref{prop:sn1}(i), $\sin\sigma\to1$ at $\alpha$, so near $\alpha$ the curve is a graph $x=X(y)$ with $V=\cot\sigma$. By Lemma~\ref{lem:IJ}, $J'=A(x)g(y)\sin^{2}\sigma$, and dividing by $y'=\sin\sigma$ gives $dJ/dy=A(X)g(y)\sin\sigma$. Since $|J|\le g(y)\to0$ at $\alpha$, $A(X(t))\sin\sigma(t)\to A(x_0)$ and $g(t)=t^{q}(1+O(t^{2}))$,
\[
J(y)=\int_0^{y}A(X(t))g(t)\sin\sigma(t)\,dt=\left(\frac{A(x_0)}{q+1}+o(1)\right)y^{q+1},
\]
hence $\cos\sigma=J/g=\left(A(x_0)/(q+1)+o(1)\right)y$ and $V=\cot\sigma=\left(A(x_0)/(q+1)+o(1)\right)y$. Hence $y^{q}V(y)\to0$ and $t^{q}G(t,X,V)$ is bounded near $0$. Integrating \eqref{ap:eq:div} from $\varepsilon$ to $y$ and letting $\varepsilon\to0$, we see that $(X,V)$ is a fixed point of $T$. Since $|A(x_0)|/(q+1)<K$, the pair belongs to $E_{\delta'}$ for some $\delta'\le\delta$, and the estimates above show that $T$ is a contraction on $E_{\delta'}$ as well. As $T$ is of Volterra type, the restriction to $[0,\delta']$ of the fixed point constructed above is the fixed point of $T$ on $E_{\delta'}$. Thus $\gamma$ coincides near $\alpha$ with the solution constructed above, and by uniqueness for \eqref{eq:flow} in $M$ the two maximal solutions agree.
\end{proof}

Combining the three kinds of endpoint with the monotonicity of $\rho$ gives the classification, and the closest point to the origin parametrizes the curves of type (c).

\begin{proof}[Proof of Theorem~\ref{thmA}]
Completeness is Corollary~\ref{thm:completude}, and $\Sigma$ is noncompact because $\gamma$ is not closed, since $I$ is strictly increasing along it by Corollary~\ref{cor:mono}, and does not have both endpoints on the axes, by Proposition~\ref{prop:cusp}. Let $\gamma$ be distinct from the Clifford ray. By Proposition~\ref{prop:endpoints} and the discussion after it, each endpoint lies on an axis or satisfies $\rho\to\infty$, and by Proposition~\ref{prop:cusp} at most one endpoint lies on an axis. If one endpoint lies on an axis, Proposition~\ref{prop:cusp} gives (a) or (b), and the topology follows from \eqref{eq:param}, an endpoint on the $x$-axis collapsing the factor $\s^{q}$, and the embeddedness from the injectivity of $\gamma$ together with the smooth closure. By Proposition~\ref{ap:existencia}, for every $x_0>0$ there is, up to orientation, exactly one maximal solution with an endpoint at $(x_0,0)$, and by Proposition~\ref{prop:cusp} it is of type (a). The same holds on the $y$-axis. If neither endpoint lies on an axis, then $\rho\to\infty$ at both endpoints, so $\rho$ attains a minimum, which is its only critical point and is nondegenerate by Corollary~\ref{cor:unimodal}. This is (c), and $\Sigma\cong\R\times\s^{p}\times\s^{q}$.

For the last assertion, fix $P=(x_0,y_0)\in\QQ$ and let $\psi_P$ be its polar angle. Since $\rho^{2}=|\gamma|^{2}$, we have
\[
\rho\rho'=\tfrac12\left(\rho^{2}\right)'=\langle\gamma,\gamma'\rangle=\rho\cos(\psi-\sigma),
\]
because $\gamma=\rho(\cos\psi,\sin\psi)$ and $\gamma'=(\cos\sigma,\sin\sigma)$. Hence $\rho'$ vanishes at a point exactly when the tangent there is orthogonal to the position vector, that is, when
\begin{equation}\label{eq:secao}
\sigma=\psi-\tfrac\pi2\qquad\text{or}\qquad\sigma=\psi+\tfrac\pi2 .
\end{equation}
Now, a curve of type (c) attains its minimal distance to the origin, since $\rho\to\infty$ at both endpoints, and it attains it at an interior point, where $\rho'=0$. By Corollary~\ref{cor:unimodal} this is its only critical point, and it is a strict global minimum. Therefore a curve of type (c) has $P$ as its closest point to the origin if and only if it passes through $P$ with the tangent condition \eqref{eq:secao} satisfied at $P$, with $\psi=\psi_P$.

\emph{Existence.} Let $\gamma$ be the trajectory through the phase point $(x_0,y_0,\psi_P-\frac\pi2)$. Along it $\rho$ has a critical point at $P$, so $\gamma$ is not the Clifford ray, along which $\rho'\equiv1$, nor of type (a) or (b), along which $\rho$ is strictly increasing by Proposition~\ref{prop:cusp}. Hence $\gamma$ is of type (c), and by Corollary~\ref{cor:unimodal} the critical point at $P$ is its only one and is the minimum of $\rho$.

\emph{Uniqueness.} If a curve of type (c) has $P$ as its closest point to the origin, then at $P$ it satisfies \eqref{eq:secao}, so its phase point there is $(x_0,y_0,\psi_P\mp\frac\pi2)$. The two choices are exchanged by the reversibility of Lemma~\ref{lem:simetrias}(i), hence they give the same curve with the two orientations, and by uniqueness of solutions of \eqref{eq:flow} the curve is the one constructed above.
\end{proof}

\section{The ends}\label{sec:ends}

Throughout this section $\gamma$ is distinct from the Clifford ray, and we consider an endpoint $\omega$ with $\rho\to\infty$, which by Theorem~\ref{thmA} is the case for every end.

\subsection{No end with one bounded coordinate}

On a subarc where the curve is a graph over the $x$-axis, the system reduces to a single first-order equation, which is the form we shall use.

\begin{lemma}[$x$ as independent variable]\label{lem:graf}
On a subarc where $\cos\sigma>0$, the function $v:=\frac{dy}{dx}=\tan\sigma$ satisfies
\begin{equation}\label{eq:v}
\frac{dv}{dx}=\left(1+v^{2}\right)\left(B(y)-v A(x)\right).
\end{equation}
\end{lemma}
\begin{proof}
$\dfrac{dv}{dx}=\dfrac{\sigma'}{\cos^{2}\sigma x'}=\dfrac{-A\sin\sigma+B\cos\sigma}{\cos^{3}\sigma}=(1+v^{2})(B-vA)$.
\end{proof}

We can now exclude the ends along which one coordinate diverges while the other stays bounded, that is, the ends asymptotic to a cylinder or to a coordinate subspace.

\begin{proposition}[No mixed ends]\label{thm:noeixo}
Let $c>0$. There is no endpoint along which $x\to\infty$ while $y$ remains bounded, and none along which $y\to\infty$ while $x$ remains bounded.
\end{proposition}
\begin{proof}
Suppose that $x\to\infty$ and $y\le Y$ at $\omega$.

\emph{Step 1.} By Corollary~\ref{cor:mono}, $\cos\sigma$ has constant sign near $\omega$, and it is positive since $x\to\infty$. Hence Lemma~\ref{lem:graf} applies.

\emph{Step 2: $\sin\sigma>0$ near $\omega$.} Otherwise $\sin\sigma<0$ near $\omega$, and since $I$ is increasing and negative, $I\ge I(s_1)=:-C$, so $|\sin\sigma|\le C/f(x)\to0$ by Lemma~\ref{lem:positivo}. Then $\cos\sigma\ge\frac12$ eventually and
\[
\sigma'=|\sin\sigma| A(x)+\cos\sigma B(y)\ge\tfrac12\cdot2\sqrt{qc}=\sqrt{qc}>0 .
\]
Since $\sigma$, lifted continuously, lies in $(-\frac\pi2,0)$ near $\omega$ and $s_+=\infty$, $\sigma$ reaches $0$ in finite arclength, which contradicts $\sin\sigma<0$.

\emph{Step 3.} By Step 2, $y$ is increasing near $\omega$, hence $y\in[y_1,Y]$ with $y_1>0$, and
\[
b_0:=\min_{[y_1,Y]}B>0 .
\]

\emph{Step 4: a barrier.} Since $x\to\infty$, we may also assume $x>\sqrt{p/c}$ near $\omega$. Let $w_0(x):=b_0/(2A(x))$, which is decreasing for $x>\sqrt{p/c}$. At a point where $v=w_0$,
\[
\frac{dv}{dx}=(1+v^{2})\left(B-w_0A\right)\ge B-\tfrac12b_0\ge\tfrac12b_0>0>\frac{dw_0}{dx},
\]
so $\{v\ge w_0\}$ is forward invariant in $x$. While $0\le v<w_0$, one has $B-vA>B-w_0A\ge\frac12b_0$, so $v$ increases at rate at least $\frac12b_0$ while $w_0$ decreases. Hence, $v$ enters $\{v\ge w_0\}$ at some finite $x_2$. Thus 
\[
\frac{dy}{dx}\ge \frac{b_0}{2A(x)},\quad\mbox{for}\quad x\ge x_2.
\]

\emph{Step 5.} For $x\ge x_2$ we have $x>\sqrt{p/c}$, hence $p/x<cx$ and $A(x)<2cx$, so $dy/dx\ge b_0/(4cx)$ and 
\[
y(x)\ge y(x_2)+\frac{b_0}{4c}\log\frac{x}{x_2}\to\infty,
\] 
contradicting $y\le Y$. The second assertion is the mirror image.
\end{proof}

Together with Proposition~\ref{prop:endpoints}, this settles the shape of every end.

\begin{corollary}[Both coordinates diverge]\label{cor:ambas}
Let $c>0$ and let $\rho\to\infty$ at $\omega$. Then $x\to\infty$, $y\to\infty$, and $\sigma\in(0,\tfrac\pi2)$ near $\omega$. In particular $w=\psi-\sigma\in(-\tfrac\pi2,\tfrac\pi2)$ and $\rho'=\cos w>0$ near $\omega$.
\end{corollary}
\begin{proof}
By Proposition~\ref{prop:endpoints}, $x$ and $y$ are eventually monotone and one of them tends to $\infty$, and by Proposition~\ref{thm:noeixo} so does the other. Hence $\cos\sigma>0$ and $\sin\sigma>0$ near $\omega$ by Corollary~\ref{cor:mono}.
\end{proof}

\subsection{Radial alignment}

By Lemma~\ref{lem:polar}, the relative phase satisfies
\begin{equation}\label{eq:wexp}
w'=-\mu(\rho)\sin w-\frac{E}{\rho},
\end{equation}
where
\[ \mu(\rho):=c\rho+\frac1\rho\quad\mbox{and}\quad \frac{E}{\rho}=\frac{q\cos\sigma}{y}-\frac{p\sin\sigma}{x}.
\]
For $c>0$ the ``equilibrium'' $w=0$ of \eqref{eq:wexp} is attracting, with rate $\mu\ge2\sqrt c$, and this is the mechanism that makes the ends conical. We write
\[
\Pi:=\frac px+\frac qy\ge\frac{|E|}{\rho}.
\]

The next result makes this precise and gives the rate at which the tangent becomes radial.

\begin{proposition}[Radial alignment]\label{prop:align}
Let $c>0$ and let $\rho\to\infty$ at $\omega$. Then $w\to0$, and there exists $s_2$ such that
\begin{equation}\label{eq:walign}
|w(s)|\le K\frac{\Pi(s)}{\rho(s)}\quad\text{for }s\ge s_2,\qquad K:=\frac\pi2+\frac{2\pi}{c}.
\end{equation}
Moreover $\psi$ has finite total variation near $\omega$ and converges to some $\psi_\infty\in[0,\frac\pi2]$.
\end{proposition}
\begin{proof}
By Corollary~\ref{cor:ambas}, near $\omega$ we have $|w|<\frac\pi2$, $x$ and $y$ increasing, so $\Pi$ is nonincreasing and tends to $0$, and $\rho$ is increasing. Let $u:=|w|$. Since $|\sin w|\ge\frac2\pi|w|$ for $|w|\le\frac\pi2$, \eqref{eq:wexp} gives, almost everywhere,
\begin{equation}\label{eq:ugron}
u'\le-\frac2\pi\mu(\rho) u+\Pi .
\end{equation}

We use the following elementary comparison twice: if $u$ is absolutely continuous and $u'\le-\alpha u+\beta$ almost everywhere on $[s_0,s]$, with constants $\alpha,\beta>0$, then, differentiating $e^{\alpha r}u(r)$,
\begin{equation}\label{eq:gronwall}
u(s)\le e^{-\alpha(s-s_0)}u(s_0)+\frac\beta\alpha .
\end{equation}

\emph{Step 1: $w\to0$.} Since $\mu\ge2\sqrt c$, \eqref{eq:gronwall} on $[s_0,s]$ gives 
\[
u(s)\le e^{-\frac{4\sqrt c}{\pi}(s-s_0)}u(s_0)+\frac{\pi\Pi(s_0)}{4\sqrt c}.
\] Letting $s\to\infty$ and then $s_0\to\infty$, we have $\limsup u=0$.

\emph{Step 2: the rate.} By Step 1, choose $s_1$ such that, for $s\ge s_1$, $\cos w\ge\frac12$, and $x,y,\rho\ge2$.

Fix $s\ge s_1+1$ and apply the comparison on the interval $[s-1,s]$, with $s$ kept fixed throughout this paragraph. For $r\in[s-1,s]$ we have $\rho(r)\ge\rho(s-1)$ and $\Pi(r)\le\Pi(s-1)$, because $\rho$ is increasing and $\Pi$ is nonincreasing near $\omega$. Hence $\mu(\rho(r))\ge c\rho(r)\ge c\rho(s-1)$, and \eqref{eq:ugron} gives
\[
u'(r)\le-\frac2\pi\mu(\rho(r))u(r)+\Pi(r)\le-\alpha_su(r)+\beta_s\qquad\text{for a.e. }r\in[s-1,s],
\]
where
\[
\alpha_s:=\frac{2c}{\pi}\rho(s-1),\qquad \beta_s:=\Pi(s-1)
\]
are constants on that interval, although they depend on the fixed endpoint $s$. Applying \eqref{eq:gronwall} on $[s-1,s]$, with $s_0=s-1$, $\alpha=\alpha_s$ and $\beta=\beta_s$, and using $u(s-1)\le\frac\pi2$, we obtain, for $r=s$,
\[
u(s)\le\frac\pi2 e^{-\alpha_s}+\frac{\beta_s}{\alpha_s}=\frac\pi2 e^{-\frac{2c}{\pi}\rho(s-1)}+\frac{\pi\Pi(s-1)}{2c\rho(s-1)} .
\]
We now compare the values at $s-1$ with those at $s$. Since $|\gamma'|=1$, the three functions $\rho$, $x$ and $y$ are $1$-Lipschitz, because $|\rho'|=|\cos w|\le1$, $|x'|=|\cos\sigma|\le1$ and $|y'|=|\sin\sigma|\le1$. Hence
\[
\rho(s-1)\ge\rho(s)-1,\qquad x(s-1)\ge x(s)-1,\qquad y(s-1)\ge y(s)-1,
\]
and, since $s\ge s_1$ gives $\rho(s),x(s),y(s)\ge2$, each of the three right-hand sides is at least half of the corresponding value at $s$:
\[
\rho(s-1)\ge\tfrac12\rho(s),\qquad x(s-1)\ge\tfrac12x(s),\qquad y(s-1)\ge\tfrac12y(s),
\]
whence $\Pi(s-1)=\frac p{x(s-1)}+\frac q{y(s-1)}\le2\Pi(s)$. Thus
\[
u(s)\le\frac\pi2 e^{-\frac c\pi\rho(s)}+\frac{2\pi}{c}\frac{\Pi(s)}{\rho(s)} .
\]
Finally, since $x,y\le\rho$, we have $\Pi/\rho\ge(p+q)/\rho^{2}$, which exceeds $e^{-\frac c\pi\rho}$ for $\rho$ large. This gives \eqref{eq:walign}.

\emph{Step 3: convergence of $\psi$.} By \eqref{eq:polar} and \eqref{eq:walign}, for $s\ge s_2$,
\[
|\psi'(s)|=\frac{|\sin w(s)|}{\rho(s)}\le\frac{|w(s)|}{\rho(s)}\le K\frac{\Pi(s)}{\rho(s)^{2}}\le K\frac{\Pi(s_2)}{\rho(s)^{2}},
\]
the last inequality because $\Pi$ is nonincreasing. On $[s_2,s_+)$ the function $\rho$ is strictly increasing, with $\rho'=\cos w\ge\frac12$ and $\rho\to\infty$, so it is a diffeomorphism of $[s_2,s_+)$ onto $[\rho(s_2),\infty)$, and the change of variable $\varrho=\rho(s)$, with $ds=d\varrho/\rho'\le2 d\varrho$, gives
\[
\int_{s_2}^{s_+}|\psi'| ds\le K\Pi(s_2)\int_{s_2}^{s_+}\frac{ds}{\rho(s)^{2}}\le2K\Pi(s_2)\int_{\rho(s_2)}^{\infty}\frac{d\varrho}{\varrho^{2}}=\frac{2K\Pi(s_2)}{\rho(s_2)}<\infty .
\]
Thus $\psi'$ is integrable on $[s_2,s_+)$, and therefore
\[
\psi(s)=\psi(s_2)+\int_{s_2}^{s}\psi'(r) dr\xrightarrow[s\to s_+]{}\psi(s_2)+\int_{s_2}^{s_+}\psi'(r) dr=:\psi_\infty ,
\]
the improper integral being absolutely convergent. Since $\psi$ takes values in $(0,\frac\pi2)$, the limit satisfies $\psi_\infty\in[0,\frac\pi2]$.
\end{proof}

The alignment now yields the decay rates along the end, once the limit direction is known to be interior.

\begin{theorem}[Conical ends]\label{thm:conico}
Let $c>0$ and let $\rho\to\infty$ at $\omega$. Then the limit angle $\psi_\infty\in(0,\frac\pi2)$, and along the end
\[
|\sin w|=O(\rho^{-2}),\qquad|\psi-\psi_\infty|=O(\rho^{-2}),\qquad \sigma'=O(\rho^{-1}),
\]
and the distance of $\gamma(s)$ to the ray $R_\infty:=\{t(\cos\psi_\infty,\sin\psi_\infty):t\ge0\}$ of angle $\psi_\infty$ satisfies
\[
\dist\left(\gamma(s),R_\infty\right)=\rho|\sin(\psi-\psi_\infty)|=O(\rho^{-1}).
\]
In particular, $\sigma\to\psi_\infty$, the end is a graph over each axis with slopes converging to $\tan\psi_\infty$ and $\cot\psi_\infty$, and $|\mathcal A|^{2}=O(\rho^{-2})$.
\end{theorem}
\begin{proof}
We first show that $\psi_\infty\in(0,\frac\pi2)$. By Corollary~\ref{cor:ambas}, $\sigma\in(0,\frac\pi2)$ near $\omega$, and $\psi\in(0,\frac\pi2)$ always, so $|w|<\frac\pi2$ and $\cos w>0$ there. Fix such an $s_0$. Then $L(s_0)=h(\psi(s_0))\cos w(s_0)>0$, and Corollary~\ref{cor:confina} gives
\[
h(\psi(s))\ge L(s_0)>0\qquad\text{for every }s\ge s_0 ,
\]
so $\psi$ remains in a compact subinterval of $(0,\frac\pi2)$ along the end, and $\psi_\infty\in(0,\frac\pi2)$.

In particular, there is $\delta>0$ with $\psi\in[\delta,\frac\pi2-\delta]$ near $\omega$, so that $x,y\ge\rho\sin\delta$ and $\Pi\le(p+q)/(\rho\sin\delta)$. Then \eqref{eq:walign} gives $|\sin w|\le|w|\le K\Pi/\rho\le C_1\rho^{-2}$, with $C_1:=K(p+q)/\sin\delta$, and therefore, by \eqref{eq:polar},
\[
|\psi'|=\frac{|\sin w|}{\rho}\le C_1\rho^{-3}.
\]
Integrating this from $s$ to $s_+$, and using again $ds\le2 d\varrho$ as in Step 3 of Proposition~\ref{prop:align},
\[
|\psi(s)-\psi_\infty|\le\int_s^{s_+}|\psi'(r)| dr\le C_1\int_s^{s_+}\frac{dr}{\rho(r)^{3}}\le2C_1\int_{\rho(s)}^{\infty}\frac{d\varrho}{\varrho^{3}}=\frac{C_1}{\rho(s)^{2}},
\]
that is, $|\psi-\psi_\infty|=O(\rho^{-2})$. Writing $u_\infty:=(\cos\psi_\infty,\sin\psi_\infty)$ and $u_\infty^{\perp}:=(-\sin\psi_\infty,\cos\psi_\infty)$, we have
\[
\gamma(s)=\rho\cos(\psi-\psi_\infty) u_\infty+\rho\sin(\psi-\psi_\infty) u_\infty^{\perp},
\]
so the orthogonal projection of $\gamma(s)$ onto the line $\R u_\infty$ is $\rho\cos(\psi-\psi_\infty) u_\infty$, and the distance from $\gamma(s)$ to that line is $\rho|\sin(\psi-\psi_\infty)|$. Since $\psi$ and $\psi_\infty$ both lie in $(0,\frac\pi2)$, we have $|\psi-\psi_\infty|<\frac\pi2$ and the projection has a positive coefficient, hence it lies on the ray $\{\psi=\psi_\infty\}$ itself. Therefore the distance from $\gamma(s)$ to that ray is exactly
\[
\rho|\sin(\psi-\psi_\infty)|\le\rho|\psi-\psi_\infty|=O(\rho^{-1}).
\] Moreover $|E|\le(p+q)/\sin\delta$, so $\sigma'=c\rho\sin w+E/\rho=O(\rho^{-1})$, and $\sigma=\psi-w\to\psi_\infty$. The slopes are $\tan\sigma$ and $\cot\sigma$. Finally, by \eqref{eq:A2},
\[
|\mathcal A|^{2}=(\sigma')^{2}+p\frac{\sin^{2}\sigma}{x^{2}}+q\frac{\cos^{2}\sigma}{y^{2}}=O(\rho^{-2}).\qedhere
\]
\end{proof}

The side from which an end approaches its asymptotic ray is governed by the function
\begin{equation}\label{eq:kappa}
\kappa(\psi):=q\cot\psi-p\tan\psi ,\qquad \psi\in(0,\tfrac\pi2).
\end{equation}
It is strictly decreasing, since $\cot$ decreases and $\tan$ increases, and it vanishes exactly at $\psi_0$, because $\tan^{2}\psi_0=q/p$. Hence $\kappa(\psi)>0$ if and only if $\psi<\psi_0$. Substituting $\sigma=\psi-w$, $x=\rho\cos\psi$ and $y=\rho\sin\psi$ into $E/\rho$ and expanding, we obtain
\[
\begin{aligned}
-\frac{E}{\rho}&=\frac{p\sin(\psi-w)}{\rho\cos\psi}-\frac{q\cos(\psi-w)}{\rho\sin\psi}\\
&=\frac1\rho\left[-\kappa(\psi)\cos w-(p+q)\sin w\right],
\end{aligned}
\]
so that \eqref{eq:wexp} becomes, with $n=p+q+1$,
\begin{equation}\label{eq:wexata}
w'=-\left(c\rho+\frac n\rho\right)\sin w-\frac{\kappa(\psi)}{\rho}\cos w .
\end{equation}
The end approaches its asymptotic ray from one side only, and the side is determined by the position of $\psi_\infty$ relative to the Clifford angle.

\begin{proposition}[The end stays on one side of $R_\infty$]\label{prop:lado}
Let $c>0$, let $\rho\to\infty$ at $\omega$ and suppose that $\psi_\infty\ne\psi_0$. Then, near $\omega$, the relative phase $w$ has a constant sign, $\psi$ is strictly monotone, and
\[
\psi<\psi_\infty\quad\text{if }\psi_\infty<\psi_0,\qquad
\psi>\psi_\infty\quad\text{if }\psi_\infty>\psi_0 .
\]
In particular $\sin(\psi-\psi_\infty)$ does not vanish near $\omega$, and the end stays strictly on one side of the ray $R_\infty$, namely on the side away from the Clifford ray.
\end{proposition}
\begin{proof}
\emph{Step 1: in the case $\psi_\infty<\psi_0$, $w<0$ for $s$ sufficiently large.} By Theorem~\ref{thm:conico}, $\psi\to\psi_\infty$ and $w\to0$, so there are $\kappa_0>0$ and $s_3$ such that, on $[s_3,s_+)$,
\[
\kappa(\psi)\ge\kappa_0>0,\qquad |w|<\frac\pi2,\qquad \cos w\ge\frac12 .
\]
We claim that $w<0$ somewhere on $[s_3,s_+)$. Otherwise $w\ge0$ there, and $\sin w\ge0$, so \eqref{eq:wexata} gives
\[
w'\le-\frac{\kappa(\psi)}{\rho}\cos w\le-\frac{\kappa_0}{2\rho}\qquad\text{on }[s_3,s_+).
\]
Since $\rho$ is $1$-Lipschitz, $\rho(s)\le\rho(s_3)+(s-s_3)$, whence
\[
\int_{s_3}^{s_+}\frac{ds}{\rho(s)}\ge\int_{s_3}^{\infty}\frac{ds}{\rho(s_3)+s-s_3}=\infty ,
\]
and therefore $w(s)\le w(s_3)-\frac{\kappa_0}2\int_{s_3}^{s}\frac{dr}{\rho(r)}\to-\infty$, contradicting $w\ge0$.

So $w(s_4)<0$ for some $s_4\ge s_3$. Moreover $w$ cannot return to $0$: at an instant where $w=0$, \eqref{eq:wexata} gives $w'=-\kappa(\psi)/\rho<0$, so $w$ crosses the value $0$ downward, and a return would require an upward crossing. Hence $w<0$ on $[s_4,s_+)$.

\emph{Step 2: conclusion.} By \eqref{eq:polar}, $\psi'=-\sin w/\rho>0$ on $[s_4,s_+)$, so $\psi$ is strictly increasing there and, converging to $\psi_\infty$, satisfies $\psi<\psi_\infty$. In particular $\psi-\psi_\infty$ has constant sign near $\omega$, and so does $\sin(\psi-\psi_\infty)$, since $|\psi-\psi_\infty|<\frac\pi2$.

The case $\psi_\infty>\psi_0$ is symmetric: there $\kappa(\psi)\le-\kappa_0<0$ near $\omega$, the same argument gives $w>0$, and $\psi$ decreases to $\psi_\infty$, so that $\psi>\psi_\infty$.
\end{proof}

The same crossing rule gives a global statement, valid along the whole curve, as long as the curve avoids the Clifford direction.

\begin{proposition}[Global version]\label{prop:lado-global}
Let $c>0$ and let $\gamma$ be of type (a) or (b), or the branch of a curve of type (c) beyond its closest point to the origin, so that $\cos w>0$ along $\gamma$.
\begin{enumerate}
\item[(i)] If $\psi<\psi_0$ along $\gamma$, then $w$ vanishes at most once, downward, and $\psi$ increases strictly to $\psi_\infty$ after that instant. If $\gamma$ is of type (a), then $w<0$ from the axis onward and $\psi$ increases along the whole curve.
\item[(ii)] If $\psi>\psi_0$ along $\gamma$, the mirror statements hold, with $w$ crossing upward, $\psi$ decreasing to $\psi_\infty$, and type (b) in place of type (a).
\end{enumerate}
\end{proposition}
\begin{proof}
At the endpoint on the axis of a curve of type (a) we have $\psi\to0$ and $w\to-\frac\pi2$, so $h(\psi)\to0$ and $L=h(\psi)\cos w\to0$. At the closest point of a curve of type (c), on either branch oriented away from that point, we have $w=\pm\frac\pi2$ and again $L=0$. In both cases $L$ is strictly increasing afterwards by Lemma~\ref{lem:L}, hence $L>0$, that is, $\cos w>0$ and $|w|<\frac\pi2$, along $\gamma$. In particular the only zeros of $\sin w$ along $\gamma$ are the zeros of $w$ itself.

Assume $\psi<\psi_0$ along $\gamma$. By \eqref{eq:wexata}, at a zero of $w$ one has $w'=-\kappa(\psi)/\rho<0$, so every zero is simple and $w$ crosses it downward, and two consecutive zeros would have opposite directions, so there is at most one. After it, $w<0$, whence $\psi'=-\sin w/\rho>0$ and $\psi$ increases. Since $\psi\to\psi_\infty$, this gives $\psi<\psi_\infty$ there. For a curve of type (a), $w$ starts at $-\frac\pi2$, and a zero of $w$ would be an upward crossing, which is excluded, hence $w<0$ along the whole curve. Item (ii) follows by applying (i) to the curve obtained from $\gamma$ by the mirror symmetry of Lemma~\ref{lem:simetrias}(ii), which sends $(x,y,\sigma;p,q)$ to $(y,x,\frac\pi2-\sigma;q,p)$, and therefore $\psi$ to $\frac\pi2-\psi$, the Clifford angle $\psi_0$ to $\frac\pi2-\psi_0$, which is the Clifford angle of the parameters $(q,p)$, and $w$ to $-w$. A curve of type (b) is sent to a curve of type (a).
\end{proof}

The sign of the mean curvature is read off from the polar angle.

\begin{lemma}[Sign of the mean curvature]\label{lem:Hpsi}
Along every solution of \eqref{eq:flow}, the mean curvature of $\Sigma$ with respect to the normal $\nu$ of Proposition~\ref{prop:sistema} is
\[
H=c\rho^{2}\psi'.
\]
In particular $H$ vanishes exactly at the critical points of $\psi$.
\end{lemma}
\begin{proof}
By Proposition~\ref{prop:sistema} and \eqref{eq:soliton},
\[
H=c\langle X,\nu\rangle=c\,(x\sin\sigma-y\cos\sigma)=-c\rho\sin w,
\]
and $\psi'=-\sin w/\rho$ by \eqref{eq:polar}.
\end{proof}

\begin{remark}\label{rem:lado-numerico}
Numerical examples suggest that the hypothesis $\psi\ne\psi_0$ of Proposition~\ref{prop:lado-global} is not automatic for curves of type (a). For $p=q=1$, $c=1$ and the curve leaving the $x$-axis at $x_0=0.2$, the polar angle appears to cross the Clifford angle $\psi_0$ once, so that $\psi$ is not monotone along the curve, and the limit angle appears to satisfy $\psi_\infty>\psi_0$. This behavior is consistent with the results above. After the crossing we have $\psi>\psi_0$, hence $\kappa(\psi)<0$ by \eqref{eq:kappa}, and by \eqref{eq:wexata} every zero of $w$ is an upward crossing. Thus $w$ vanishes at most once after the crossing, and from then on $\psi$ decreases to $\psi_\infty$, approaching it from the side away from the Clifford ray, as in Proposition~\ref{prop:lado}. Figure~\ref{fig:lado} shows this curve, together with the functions $\psi-\psi_0$ and $w$ along it.
\end{remark}

\begin{figure}[ht]
\centering
\includegraphics[width=\textwidth]{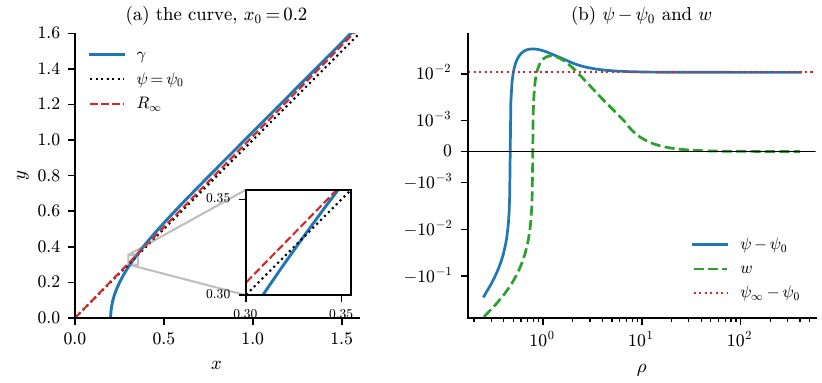}
\caption{The curve of Remark~\ref{rem:lado-numerico}. In (a) the curve, the Clifford ray and the asymptotic ray $R_\infty$, which differ in direction by about $10^{-2}$ radians. The inset enlarges the crossing at $\rho\approx0.47$. In (b) the functions $\psi-\psi_0$ and $w$, on a symmetric logarithmic scale.}
\label{fig:lado}
\end{figure}

The two asymptotic directions of a curve of type (c) are now identified as the extreme values of $L$, which bounds them in terms of any point of the curve.

\begin{corollary}[The asymptotic angles are the extremes of $L$]\label{cor:extremos}
Let $c>0$ and let $\gamma$ be of type (c), with limit angles $\psi_\alpha$ at $\alpha$ and $\psi_\omega$ at $\omega$. Then
\[
h(\psi_\omega)=\sup_sL(s),\qquad h(\psi_\alpha)=-\inf_sL(s),
\]
and consequently
\[
-h(\psi_\alpha)\le L(s)\le h(\psi_\omega)\qquad\text{for every }s .
\]
In particular, at every instant where the tangent is radial one has $h(\psi)\le h(\psi_\omega)$, and at every instant where it is anti-radial, $h(\psi)\le h(\psi_\alpha)$.
\end{corollary}
\begin{proof}
By Theorem~\ref{thm:conico}, $w\to0$ and $\psi\to\psi_\omega$ as $s\to s_+$, so $L=h(\psi)\cos w\to h(\psi_\omega)$. At the other endpoint, the reversibility of Lemma~\ref{lem:simetrias}(i) turns the end $\alpha$ into an end of the reversed trajectory, along which the tangent becomes radial as well, and in the original orientation this means $w\to\pm\pi$ and $\psi\to\psi_\alpha$ as $s\to s_-$, so that $L\to-h(\psi_\alpha)$. Since $L$ is nondecreasing by Lemma~\ref{lem:L}, these two limits are its infimum and its supremum, and the displayed bounds follow. The last two assertions are the cases $\cos w=1$ and $\cos w=-1$ of $L=h(\psi)\cos w$.
\end{proof}

The description of the ends also yields the properness of $\Sigma$ and its volume growth.

\begin{proof}[Proof of Theorem~\ref{thmB}]
Items (i)--(iv) were obtained in Section~\ref{sec:ends}. Item (i) is Corollary~\ref{cor:ambas}, which rests on Proposition~\ref{thm:noeixo}. Item (ii) is Theorem~\ref{thm:conico}, item (iii) follows from Corollary~\ref{cor:ambas} and Theorem~\ref{thm:conico}, and item (iv) is Proposition~\ref{prop:lado}.

Item (v) is immediate for the curves that leave an axis. For a curve of type (a), oriented away from the axis, the hypothesis $\psi\ne\psi_0$ implies $\psi<\psi_0$, and Proposition~\ref{prop:lado-global}(i) gives $\psi'>0$. Likewise, on a curve of type (b) oriented away from the axis, $\psi>\psi_0$ and Proposition~\ref{prop:lado-global}(ii) gives $\psi'<0$.

The curves of type (c) require more care. Let $P$ be the unique point of minimum distance to the origin, and orient each branch away from $P$. Then $w(P)=\frac\pi2$ on one branch and $w(P)=-\frac\pi2$ on the other. Suppose first that $\psi<\psi_0$ along the whole curve. By \eqref{eq:wexata}, at every zero of $w$,
\[
w'=-\frac{\kappa(\psi)}{\rho}<0,
\]
so every zero is a downward crossing and there is at most one on each branch. On the branch starting with $w=-\frac\pi2$, no such crossing is possible, hence $w<0$ and $\psi'>0$ throughout.

On the branch starting with $w=\frac\pi2$, a zero must occur. Otherwise $w>0$ and $\psi'<0$ throughout, so that $\psi_\infty\le\psi(P)<\psi_0$, whereas Proposition~\ref{prop:lado} gives $w<0$ near that end. Thus $w$ crosses zero exactly once on this branch, from positive to negative values, and $\psi$ decreases before the crossing and increases afterwards. The crossing is therefore a critical point $s_*$ of $\psi$, and it is the unique global minimum of $\psi$ on the whole curve. Differentiating $\psi'=-\sin w/\rho$ at $s_*$, where $w=0$, we obtain
\[
\psi''(s_*)=-\frac{w'(s_*)}{\rho(s_*)}=\frac{\kappa(\psi(s_*))}{\rho(s_*)^{2}}>0,
\]
so the minimum is nondegenerate. If $\psi>\psi_0$, the same argument with reversed signs shows that $\psi$ has exactly one critical point on the whole curve, a nondegenerate global maximum.

Since $\psi'$ has a constant sign along curves of types (a) and (b), and changes sign at the nondegenerate critical point along a curve of type (c), Lemma~\ref{lem:Hpsi} gives the assertions on the mean curvature.

As for properness, $\Sigma$ has at most two ends by Theorem~\ref{thmA}, along each of which $\rho\to\infty$, and the rest of the profile curve is a compact arc, possibly closed off at an endpoint on an axis. Hence the preimage of a compact set is compact, and $\Sigma$ is proper.

It remains to compute the volume growth. The parametrization \eqref{eq:param} is orthogonal, with $|\partial_sX|=1$, so that the induced metric is
\[
ds^{2}+x(s)^{2} g_{\s^{p}}+y(s)^{2} g_{\s^{q}}
\]
and the volume element is $x^{p}y^{q} ds\, d\mu_{\s^{p}}d\mu_{\s^{q}}$. Hence
\[
\mathrm{vol}\left(\Sigma\cap B_R\right)=\omega_p\omega_q\int_{\{\rho\le R\}}x(s)^{p}y(s)^{q} ds ,
\]
where $\omega_p$ and $\omega_q$ are the volumes of the unit spheres. On the $j$-th end, Theorem~\ref{thm:conico} gives $\psi=\psi_j+O(\rho^{-2})$ and $w=O(\rho^{-2})$, with $\psi_j\in(0,\frac\pi2)$. Since $\rho'=\cos w>0$ sufficiently far along that end, we may use $\rho$ as a parameter, obtaining
\[
\begin{aligned}
x^py^q\,ds
&=\rho^{p+q}\frac{\cos^p\psi\,\sin^q\psi}{\cos w}\,d\rho\\
&=\rho^{n-1}\left(\cos^p\psi_j\,\sin^q\psi_j+O(\rho^{-2})\right)\,d\rho.
\end{aligned}
\]
Since $n=p+q+1\ge3$, integration along this end up to radius $R$ gives
\[
\int_{\text{$j$-th end},\,\rho\le R} x^py^q\,ds
=\frac{\cos^p\psi_j\,\sin^q\psi_j}{n}R^n+O(R^{n-2}).
\]
The contribution from the remaining compact part of the profile curve is bounded. Summing over the $k$ ends, where $k=1$ for types (a) and (b) and $k=2$ for type (c), yields
\[
\mathrm{vol}(\Sigma\cap B_R)
=\frac{\omega_p\omega_q}{n}\left(\sum_{j=1}^{k}\cos^p\psi_j\,\sin^q\psi_j\right)R^n
+O(R^{n-2}).
\]
Dividing by $R^n$ and taking $R\to\infty$ proves the asserted limit, which is positive because $\psi_j\in(0,\frac\pi2)$ for every end.
\end{proof}

\section{Loops and genericity}\label{sec:loops}

\subsection{Loops}\label{sub:loops}

We first collect the elementary facts about plane curves that will be used, stated for the profile curves of \eqref{eq:flow}.

\begin{lemma}[Loops of a profile curve]\label{lem:laco}
Let $\gamma$ be a profile curve of \eqref{eq:flow}.
\begin{enumerate}
\item[(i)] Let $\gamma|_{[a,b]}$ be a loop, that is, a simple subarc with $\gamma(a)=\gamma(b)$, and let $\theta\in(-\pi,\pi]$ be its exterior angle at the corner, the angle from $\gamma'(b)$ to $\gamma'(a)$. Then $\theta\ne\pi$, and
\[
\Delta\sigma+\theta=\pm2\pi ,
\]
where $\Delta\sigma=\sigma(b)-\sigma(a)$ and the sign is that of the orientation of the loop.
\item[(ii)] Let $\gamma|_{[t_1,t_2]}$ be a subarc along which the tangent angle stays in an open interval $(\alpha,\alpha+\pi)$. Then $\gamma(t_2)\ne\gamma(t_1)$ and the direction of $\gamma(t_2)-\gamma(t_1)$ lies in that interval.
\item[(iii)] If $c>0$ and $\gamma$ has a self-intersection, then it has a loop.
\end{enumerate}
\end{lemma}
\begin{proof}
\emph{(i).} The displayed identity is the rotation index theorem for a simple closed piecewise-$C^{1}$ curve with one corner, applied to $\gamma|_{[a,b]}$. See \cite[Section~4-5, Theorem of Turning Tangents]{doCarmoCurvas} or \cite{Hopf}. It remains to exclude $\theta=\pi$, that is, to show that the two tangents at the corner are never opposite. Write $\Gamma=(x,y,\sigma)$ for the solution of \eqref{eq:flow} whose projection to $\QQ$ is $\gamma$, a curve in the phase space $M$. Suppose the tangents at the corner were opposite, so that $\sigma(b)=\sigma(a)+\pi$ modulo $2\pi$ and hence
\[
\Gamma(b)=\mathcal R\left(\Gamma(a)\right),\qquad \mathcal R(x,y,\sigma):=(x,y,\sigma+\pi),
\]
since $\gamma(a)=\gamma(b)$. By Lemma~\ref{lem:simetrias}(i), $t\mapsto \mathcal R\left(\Gamma(-t)\right)$ is again a solution of \eqref{eq:flow}, and the two solutions $t\mapsto\Gamma(b+t)$ and $t\mapsto \mathcal R\left(\Gamma(a-t)\right)$ agree at $t=0$, so they agree for every admissible $t$, by uniqueness. Taking $t=(a-b)/2$ and writing $m:=(a+b)/2$ gives $\Gamma(m)=\mathcal R\left(\Gamma(m)\right)$, that is, $\sigma(m)=\sigma(m)+\pi$ modulo $2\pi$, which is absurd.

\emph{(ii).} Let $u$ be the unit vector of angle $\alpha+\frac\pi2$, so that $\langle v,u\rangle>0$ for every nonzero vector whose direction lies in $(\alpha,\alpha+\pi)$. Since $\gamma'(s)$ has this property for every $s\in[t_1,t_2]$,
\[
\left\langle\gamma(t_2)-\gamma(t_1),u\right\rangle=\int_{t_1}^{t_2}\left\langle\gamma'(s),u\right\rangle ds>0 ,
\]
so the chord is nonzero. Moreover the vectors $\gamma'(s)$ lie in the open convex cone spanned by the directions of $(\alpha,\alpha+\pi)$, and that cone is closed under sums and under multiplication by positive scalars, hence it contains the integral $\gamma(t_2)-\gamma(t_1)$, that is, the direction of the chord lies in $(\alpha,\alpha+\pi)$.

\emph{(iii).} Let $\gamma(u_0)=\gamma(v_0)$ with $u_0<v_0$. Since $|\gamma'|=1$ and the curvature $\sigma'$ is bounded on the compact interval $[u_0,v_0]$, there is $\eta>0$ such that $\gamma$ is injective on every subinterval of $[u_0,v_0]$ of length at most $\eta$, and in particular $v_0-u_0\ge\eta$. The set
\[
\mathcal K:=\left\{(u,v):u_0\le u<v\le v_0,\ v-u\ge\eta,\ \gamma(u)=\gamma(v)\right\}
\]
is therefore compact and nonempty, so it contains a pair $(u,v)$ minimizing $v-u$. The subarc $\gamma|_{[u,v]}$ is closed, and it is simple: a self-intersection $\gamma(u')=\gamma(v')$ with $u\le u'<v'\le v$ and $(u',v')\ne(u,v)$ would satisfy $v'-u'\ge\eta$, by the choice of $\eta$, and would therefore belong to $\mathcal K$ with $v'-u'<v-u$. It remains to check that the endpoint tangents are distinct. If $\gamma'(u)=\gamma'(v)$, then the phase points agree, $\Gamma(u)=\Gamma(v)$ in $M$. Uniqueness for \eqref{eq:flow} would make the corresponding trajectory periodic with period $v-u$, contradicting the strict monotonicity of $I$ and $J$ for $c>0$ (Corollary~\ref{cor:mono}). Thus $\gamma'(u)\ne\gamma'(v)$, and $\gamma|_{[u,v]}$ is a loop.
\end{proof}

The first consequence of Lemma~\ref{lem:laco}(i) is that a loop forces the tangent angle to sweep more than half a turn, and therefore to cross the two axes of the plane of directions.

\begin{lemma}[Signs along a loop]\label{lem:sinais}
Let $c>0$ and let $\gamma|_{[a,b]}$ be a loop. Then the unique zeros $t_I$ of $\sin\sigma$ and $t_J$ of $\cos\sigma$ lie in $(a,b)$, and
\begin{equation}\label{eq:sinais}
\sin\sigma(a)<0<\sin\sigma(b),\qquad\cos\sigma(a)<0<\cos\sigma(b).
\end{equation}
\end{lemma}
\begin{proof}
By Lemma~\ref{lem:laco}(i), $|\Delta\sigma|=|\pm2\pi-\theta|>\pi$ on $[a,b]$. The open interval between $\sigma(a)$ and $\sigma(b)$ has length $>\pi$, so it contains a zero of $\sin$ and a zero of $\cos$, both attained at interior instants. By Corollary~\ref{cor:mono} these are the unique zeros $t_I$, $t_J$, and the signs change from $-$ to $+$.
\end{proof}

Figure~\ref{fig:loop} shows two curves of type (c) with a loop, together with the region bounded by it and the point closest to the origin.

The proof of Theorem~\ref{thmC} is now a matter of assembling these facts: the signs along a loop locate it, the cone lemma excludes a second one, and Green's theorem turns the turning of the tangent into an integral over the region bounded by the loop.

\begin{figure}[t]
\centering
\includegraphics[width=\textwidth]{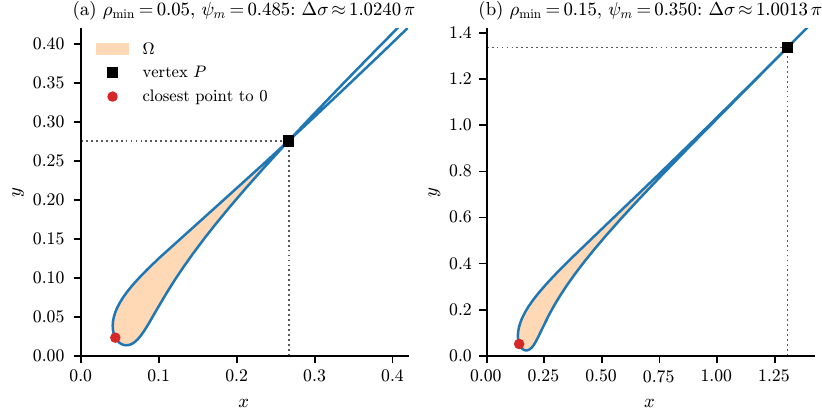}
\caption{Two loops, for $p=q=1$, $c=1$ (Theorem~\ref{thmC}), on the curves of type (c) whose closest points to the origin have polar coordinates $(\rho,\psi)=(0.05,0.485)$ in (a) and $(0.15,0.350)$ in (b). The loop passes through the closest point to the origin, and the region it bounds (shaded) lies in the box $\{x<x_P, y<y_P\}$ (dotted). The turning along it is $1.0240\pi$ in (a) and $1.0013\pi$ in (b). In (b) the two tangents at $P$ are almost opposite.}
\label{fig:loop}
\end{figure}

\begin{proof}[Proof of Theorem~\ref{thmC}]
\emph{Step 1: the curves of types (a) and (b) are simple.} This is Proposition~\ref{prop:cusp}.

\emph{Step 2: a curve of type (c) has at most one loop.} Let $\gamma|_{[a_1,b_1]}$ and $\gamma|_{[a_2,b_2]}$ be two distinct loops, $a_1\le a_2$. By Lemma~\ref{lem:sinais}, $t_I$ belongs to both intervals, so they are not disjoint. They are not nested, since $\gamma|_{[a_1,b_1]}$ is simple. Hence $a_1<a_2<b_1<b_2$, and $t_I,t_J\in(a_2,b_1)$. On $[a_1,a_2]$ we then have $\sin\sigma<0$ and $\cos\sigma<0$, so $\sigma$ stays in an open quadrant, and by Lemma~\ref{lem:laco}(ii) the chord $Q-P$, with $P=\gamma(a_1)=\gamma(b_1)$ and $Q=\gamma(a_2)=\gamma(b_2)$, is nonzero and points into the open third quadrant. On $[b_1,b_2]$, $\sin\sigma>0$ and $\cos\sigma>0$, and the same chord $Q-P$ points into the open first quadrant. This is absurd. ($Q\ne P$, since $\gamma(a_1)=\gamma(a_2)$ with $a_1<a_2<b_1$ would contradict the simplicity of $\gamma|_{[a_1,b_1]}$.)

\emph{Step 3: proof of item (i).} By Corollary~\ref{cor:unimodal}, $\rho$ is strictly decreasing before its minimum point $m$ and strictly increasing after it. Since $\rho(a)=\rho(b)$, we get $a<m<b$. By Lemma~\ref{lem:sinais}, $t_J\in(a,b)$ and $\cos\sigma$ changes sign there, from $-$ to $+$, and, by Corollary~\ref{cor:mono}, it has no other zero, so $\cos\sigma<0$ on $[a,t_J)$ and $\cos\sigma>0$ on $(t_J,b]$. As $x'=\cos\sigma$, the function $x$ is therefore strictly decreasing on $[a,t_J]$ and strictly increasing on $[t_J,b]$, so it attains its maximum on $[a,b]$ at one of the two endpoints. Since $x(a)=x(b)=x_P$, that maximum is $x_P$, and $x<x_P$ at every interior instant. The same argument with $\sin\sigma$ and $t_I$ gives $y<y_P$ on $(a,b)$.

\emph{Step 4: proof of items (ii) and (iii).} Since $\sigma'=-A(x)y'+B(y)x'$, Green's theorem gives
\[
\begin{aligned}
\Delta\sigma&=\oint_{\partial\Omega}\left(B dx-A dy\right)=\varepsilon\iint_\Omega\left(-A'(x)-B'(y)\right)dx dy\\
&=\varepsilon\iint_\Omega\left(\frac p{x^{2}}+\frac q{y^{2}}-2c\right)dx dy,
\end{aligned}
\]
where $\varepsilon=\pm1$ is the orientation of $\partial\Omega$. By Lemma~\ref{lem:laco}(i), $\Delta\sigma=2\pi\varepsilon-\theta$ with $|\theta|<\pi$, so $\sgn\Delta\sigma=\varepsilon$ and the integral equals $|\Delta\sigma|>\pi$. On the other hand, by \eqref{eq:sinais}, $\sigma(a)$ lies in the open third quadrant and $\sigma(b)$ in the open first quadrant, and on $[a,b]$ the angle $\sigma$ crosses exactly one multiple of $\pi$ and exactly one odd multiple of $\frac\pi2$, hence $|\Delta\sigma|<\frac{3\pi}2$. Finally, suppose that $\Omega\subset[\sqrt{p/c},\infty)\times[\sqrt{q/c},\infty)$. There $p/x^{2}\le c$ and $q/y^{2}\le c$, so the integrand is nonpositive and the integral is at most $0$, contradicting $\varepsilon\Delta\sigma=|\Delta\sigma|>\pi$. Hence $\Omega$ contains a point with $x<\sqrt{p/c}$ or with $y<\sqrt{q/c}$. Since a nonconstant linear function attains its minimum on the compact set $\overline\Omega$ only at points of $\partial\Omega$, the loop contains such a point as well.

\emph{Step 5: the turning along the whole curve.} Let $\gamma$ be of type (c) and let $\Delta\sigma:=\lim_{s\to s_+}\sigma-\lim_{s\to s_-}\sigma$. By Theorem~\ref{thm:conico}, $\sigma\to\psi_\omega$ at $\omega$ and, by reversibility, $\sigma\to\psi_\alpha+\pi$ at $\alpha$, modulo $2\pi$. Thus $\sin\sigma$ and $\cos\sigma$ are negative near $\alpha$ and positive near $\omega$, so by Corollary~\ref{cor:mono} each changes sign exactly once, and $\sigma$ crosses exactly one multiple of $\pi$ and one odd multiple of $\frac\pi2$. Hence $\Delta\sigma=\psi_\omega-\psi_\alpha+\pi$ if $\sigma$ increases through $\frac{3\pi}2$ and $2\pi$, and $\Delta\sigma=\psi_\omega-\psi_\alpha-\pi$ if it decreases through $\pi$ and $\frac\pi2$.
\end{proof}

We close this subsection with the questions left open by Theorem~\ref{thmC} and Proposition~\ref{prop:lado-global}.

\begin{remark}\label{rem:questoes}
Theorem~\ref{thmC} bounds the number of loops, not the number of self-intersection points: the two ends of a curve of type (c) may cross each other several times outside the loop. The classification thus leaves open the number of crossings of the two branches of a curve of type (c) and the number of curves of type (a) with a prescribed limit angle. By Corollary~\ref{cor:unimodal}, each branch of a curve of type (c), starting at its closest point to the origin, is a polar graph $\psi=\psi_\pm(\rho)$, so that the curve is embedded if and only if $\psi_+(\rho)\ne\psi_-(\rho)$ for every $\rho$ beyond the minimum.

Numerical examples suggest that both phenomena occur only at small scales, and that the second one occurs only near the Clifford ray. For $p=q=1$ and $c=1$, they suggest that a curve of type (c) is embedded when its closest point to the origin lies at distance at least $0.25$ from it, and has self-intersections when this distance is at most $0.1$, their number increasing as the distance tends to $0$. They also suggest that the limit angle $\psi_\infty(x_0)$ of the curve of type (a) leaving the $x$-axis at $x_0$ is strictly decreasing for $x_0\ge0.27$, with $\psi_\infty(x_0)\to0$ as $x_0\to\infty$, so that every angle in $(0,\psi_0-4\cdot10^{-4})$ is the limit angle of exactly one curve of type (a), while uniqueness fails in a window of width about $10^{-2}$ around $\psi_0$. By homothety, these distances scale with $c^{-1/2}$.

Near the origin, in the variable $t=\int ds/\rho$, the equations \eqref{eq:polar} and \eqref{eq:wexata} become
\[
\frac{d\psi}{dt}=-\sin w,\qquad \frac{dw}{dt}=-(n+c\rho^{2})\sin w-\kappa(\psi)\cos w,\qquad \frac{d\rho}{dt}=\rho\cos w,
\]
with $\kappa$ as in \eqref{eq:kappa}. On the invariant plane $\{\rho=0\}$ this is the system of the minimal case $c=0$, whose bi-rotational solutions were studied in \cite{Alencar93,ABPRS}. The linearization of this system at the Clifford direction $(\psi,w)=(\psi_0,0)$ is
\[
\ddot\xi+n\dot\xi+2(n-1)\xi=0,\qquad \xi=\psi-\psi_0,
\]
with characteristic polynomial $\lambda^{2}+n\lambda+2(n-1)$. Its roots are complex exactly when $n\le6$, that is, when $p+q\le5$, which is the range in which $\mathcal C$ is unstable \cite{Simons}. In that range the linearized solutions oscillate around the Clifford direction. This suggests the following questions.
\begin{enumerate}
\item[(i)] For $p+q\le5$, are there infinitely many curves of type (a) asymptotic to the Clifford cone, and is the number of self-intersections of the curves of type (c) unbounded?
\item[(ii)] For $p+q\ge6$, does the polar angle of every curve of type (a) stay below $\psi_0$, so that Proposition~\ref{prop:lado-global} applies to all of them?
\item[(iii)] Is every curve of type (c) whose closest point to the origin lies far enough from it embedded, and is $x_0\mapsto\psi_\infty(x_0)$ strictly monotone for $x_0$ large?
\end{enumerate}
\end{remark}

\subsection{Genericity}\label{sec:generic}

For every $c$, the flow of \eqref{eq:flow} preserves $\mathfrak m=\Phi dx dy d\sigma$. Indeed, write the field of \eqref{eq:flow} as
\[
V(x,y,\sigma)=\left(\cos\sigma,\ \sin\sigma,\ -A(x)\sin\sigma+B(y)\cos\sigma\right)
\]
on $M=\QQ\times\s^1$, with coordinates $(x,y,\sigma)$. Since $\Phi=f(x)g(y)$ with $f'=Af$ and $g'=Bg$, we have $\partial_x\Phi=A\Phi$, $\partial_y\Phi=B\Phi$ and $\partial_\sigma\Phi=0$, so that
\[
\begin{aligned}
\dive(\Phi V)&=\partial_x\left(\Phi\cos\sigma\right)+\partial_y\left(\Phi\sin\sigma\right)+\partial_\sigma\left(\Phi\left(-A\sin\sigma+B\cos\sigma\right)\right)\\
&=A\Phi\cos\sigma+B\Phi\sin\sigma+\Phi\left(-A\cos\sigma-B\sin\sigma\right)=0.
\end{aligned}
\]
Since a flow preserves $\varrho dV$ whenever $\dive(\varrho V)=0$, we conclude that $\mathfrak m$ is invariant. For $c>0$ its total mass is infinite, but it is finite on bounded sets, which is all we need.

We denote by $\varphi_t$ the local flow of \eqref{eq:flow} on $M$, by $(s_-(v),s_+(v))$ the maximal interval of the solution with $\varphi_0(v)=v$, and by $x(v),y(v),\rho(v)$ the corresponding coordinates of the point $v\in M$. Set
\[
C^{\pm}:=\{v\in M:\ s_\pm(v)\text{ is finite}\},
\]
the set of initial conditions whose trajectory ceases in finite arclength in the corresponding time direction. These sets are Borel: if $s_+(v_0)>T$, the solution through $v_0$ is defined on $[0,T]$ and its trace is compact, so by continuous dependence the same holds for every $v$ near $v_0$. Thus $\{s_+>T\}$ is open, $s_+$ is lower semicontinuous, and $C^{+}=\bigcup_{j\in\N}\{s_+\le j\}$ is a countable union of closed sets. The same argument applies to $s_-$.

The trajectories that reach the boundary of $\QQ$ in finite time are the ones to be discarded, and they form a null set, by a collar estimate.

\begin{lemma}[Incomplete trajectories are negligible]\label{lem:nulo}
Let $c>0$. Then $\mathfrak m(C^{+})=\mathfrak m(C^{-})=0$.
\end{lemma}
\begin{proof}
Let $k:=\min(p,q)$ and, for $R>0$ and $\varepsilon\in(0,1)$, let 
\[
G_{\varepsilon,R}:=\{u\in C^{-}:-s_-(u)\le\varepsilon,\rho(u)\le R\}.
\]
Every $u\in G_{\varepsilon,R}$ satisfies $x(u)\le\varepsilon$ or $y(u)\le\varepsilon$. Indeed, as in the proof of Proposition~\ref{prop:endpoints}, the profile curve $\gamma_u$ of $u$, with $\gamma_u(0)=(x(u),y(u))$, converges as $s\to s_-(u)$ to a point $P\in\partial\QQ$, and
\[
\left|P-(x(u),y(u))\right|=\lim_{s\to s_-(u)}\left|\gamma_u(s)-\gamma_u(0)\right|\le\lim_{s\to s_-(u)}|s|=-s_-(u)\le\varepsilon ,
\]
again because $|\gamma_u'|=1$. Hence
\[
G_{\varepsilon,R}\subset
\left\{v\in M: x(v)\le\varepsilon,\ \rho(v)\le R\right\}\cup
\left\{v\in M: y(v)\le\varepsilon,\ \rho(v)\le R\right\},
\]
and the two sets on the right are contained in $(0,\varepsilon]\times(0,R]\times\s^{1}$ and in $(0,R]\times(0,\varepsilon]\times\s^{1}$, respectively. Integrating $\Phi=f(x)g(y)$ over these,
\[
\mathfrak m(G_{\varepsilon,R})\le2\pi\left(\int_0^{\varepsilon}f dx\int_0^{R}g dy+\int_0^{R}f dx\int_0^{\varepsilon}g dy\right)\le K_R\varepsilon^{k+1},
\]
where $K_R$ depends only on $p,q,c$ and $R$: for $0<\varepsilon\le1$ one has $\int_0^{\varepsilon}f dx\le e^{c/2}\varepsilon^{p+1}/(p+1)$ and $\int_0^{\varepsilon}g dy\le e^{c/2}\varepsilon^{q+1}/(q+1)$, while the integrals over $[0,R]$ are finite.

Let $T,R>0$ and 
\[
C^{-}_{T,R}:=\{v\in C^{-}:-s_-(v)\le T,\rho(v)\le R\}.
\] 
For $v\in C^{-}_{T,R}$ and $j$ with $-s_-(v)\in(j\varepsilon,(j+1)\varepsilon]$, the point $u=\varphi_{-j\varepsilon}(v)$ is defined, satisfies $-s_-(u)=-s_-(v)-j\varepsilon\le\varepsilon$ and $s_+(u)>j\varepsilon$, so that $\varphi_{j\varepsilon}(u)=v$, and
\[
\rho(u)\le\rho(v)+\left|\gamma_v(-j\varepsilon)-\gamma_v(0)\right|\le R+j\varepsilon\le R+T ,
\]
because $|\gamma_v'|=1$. Thus, $u\in G_{\varepsilon,R+T}$. Hence $C^{-}_{T,R}$ is covered by the $N:=\lceil T/\varepsilon\rceil$ sets $\varphi_{j\varepsilon}(G_{\varepsilon,R+T}\cap\{s_+>j\varepsilon\})$, each of which has the same $\mathfrak m$-measure as a subset of $G_{\varepsilon,R+T}$, by the invariance of $\mathfrak m$. Therefore
\[
\mathfrak m(C^{-}_{T,R})\le\left(\frac T\varepsilon+1\right)K_{R+T}\varepsilon^{k+1}\xrightarrow[\varepsilon\to0]{}0 ,
\]
because $k+1\ge2$.
Taking the union over $T,R\in\N$ gives $\mathfrak m(C^{-})=0$. Finally, the map $\mathcal R(x,y,\sigma)=(x,y,\sigma+\pi)$ preserves $\mathfrak m$, because $\Phi$ does not depend on $\sigma$, and maps $C^{-}$ onto $C^{+}$, by the reversibility of Lemma~\ref{lem:simetrias}(i). Therefore, $\mathfrak m(C^{+})=0$ as well.
\end{proof}

We conclude the paper with the proof of Theorem~\ref{thmD}. The initial conditions of type (c) are those whose trajectory meets the critical points of $\rho$, which form a surface transversal to the flow by Corollary~\ref{cor:unimodal}, and the complement is null by Lemma~\ref{lem:nulo}.

\begin{proof}[Proof of Theorem~\ref{thmD}]
Consider the function
\[
\eta(v):=\left\langle\gamma,\gamma'\right\rangle=\rho\rho'=x\cos\sigma+y\sin\sigma =\rho\cos(\sigma-\psi),
\]
defined for $v=(x,y,\sigma)\in M$, where $\psi$ is the polar angle of $(x,y)$, and let $Z:=\{\eta=0\}$. Thus, $Z=S\cup \mathcal R(S)$, where
\[
S:=\left\{(x,y,\sigma)\in M:\sigma=\psi-\tfrac\pi2\right\},\qquad \mathcal R(x,y,\sigma):=(x,y,\sigma+\pi).
\]

\emph{Step 1: $Z$ is transversal to the flow.} Along a solution, $\eta=(\rho^{2}/2)'$ and, by Corollary~\ref{cor:unimodal},
\[
\frac{d}{dt}\,\eta(\varphi_t(v))=\left(\frac{\rho^{2}}2\right)''=n+c\rho^{2}>0
\]
at every point of $Z$. The left-hand side is the derivative of $\eta$ in the direction of the field $V$ of \eqref{eq:flow}, that is, $\left\langle\nabla\eta,V\right\rangle$. Hence $\nabla \eta\ne0$ at every point of $Z$, so $0$ is a regular value of $\eta$ and $Z=\eta^{-1}(0)$ is a smooth surface, and $V$ is nowhere tangent to $Z$, since $\left\langle\nabla \eta,V\right\rangle>0$ there, so $Z$ is transversal to the flow. Moreover $\eta$ is strictly increasing at each of its zeros along a trajectory, so each trajectory meets $Z$ at most once.

\emph{Step 2: the flow-out of $Z$ is open.} Let $U:=\{\varphi_t(z):z\in Z,\ t\in(s_-(z),s_+(z))\}$, which is the set of $v\in M$ whose trajectory meets $Z$. Define
\[
F(t,v):=\eta(\varphi_t(v))
\]
on the domain of the flow, an open subset of $\R\times M$ on which $F$ is smooth. Let $v_0\in U$ and let $t_0$ be such that $F(t_0,v_0)=0$, that is, $\varphi_{t_0}(v_0)\in Z$. By Step 1,
\[
\partial_tF(t_0,v_0)=n+c\rho\left(\varphi_{t_0}(v_0)\right)^{2}>0 ,
\]
so the implicit function theorem gives a neighborhood $W$ of $v_0$ in $M$ and a smooth function $\tau\colon W\to\R$ with $\tau(v_0)=t_0$ and $F(\tau(v),v)=0$ for every $v\in W$. Thus $\varphi_{\tau(v)}(v)\in Z$, that is, $W\subset U$, and $U$ is open.

\emph{Step 3: $U$ is the set of initial conditions of type (c).} A trajectory meets $Z$ exactly when $\rho$ has a critical point along it. On the Clifford trajectory $\rho'$ never vanishes (it equals $\pm1$, according to the orientation), and on a trajectory of type (a) or (b) $\rho'$ never vanishes either: when the curve is oriented away from the axis, $\rho$ is strictly increasing by Proposition~\ref{prop:cusp}, and reversing the orientation only changes the sign of $\rho'$. Thus these trajectories avoid $Z$. Conversely, a trajectory with a critical point of $\rho$ is of type (c), as shown in the proof of Theorem~\ref{thmA}.

\emph{Step 4: the complement is null.} By Theorem~\ref{thmA}, $M\setminus U$ consists of the Clifford ray and of the trajectories of types (a) and (b). Each of them has an endpoint reached in finite arclength, the origin in the first case and a point of one of the axes in the other two, so $M\setminus U\subset C^{+}\cup C^{-}$, which is $\mathfrak m$-null by Lemma~\ref{lem:nulo}. Hence $U$ is open and of full $\mathfrak m$-measure.
\end{proof}

\end{document}